\documentclass[a4paper, 11pt]{article}
\usepackage{amsfonts, amsmath, amssymb, amsthm, fullpage, graphicx, tikz-cd}
\usepackage{hyperref}
\hypersetup{
	colorlinks,
	citecolor=blue,
	linkcolor=blue,
	urlcolor=green}

\newtheorem{theorem}{Theorem}[section]
\newtheorem{lem}[theorem]{Lemma}
\newtheorem{prop}[theorem]{Proposition}
\newtheorem{rem}[theorem]{Remark}
\newtheorem{definition}[theorem]{Definition}

\begin{document}

\title{\vspace{-30pt}Scaling limit for the cover time\\
of the $\lambda$-biased random walk\\
on a binary tree with $\lambda<1$}
\author{D.\ A.\ Croydon\footnote{Research Institute for Mathematical Sciences, Kyoto University, croydon@kurims.kyoto-u.ac.jp.}}

\maketitle

\begin{abstract}
The $\lambda$-biased random walk on a binary tree of depth $n$ is the continuous-time Markov chain that has unit mean holding times and, when at a vertex other than the root or a leaf of the tree in question, has a probability of jumping to the parent vertex that is $\lambda$ times the probability of jumping to a particular child. (From the root, it chooses one of the two children with equal probability.) For this process, when $\lambda<1$, we derive an $n\rightarrow \infty$ scaling limit for the cover time, that is, the time taken to visit every vertex. The distributional limit is described in terms of a jump process on a Cantor set that can be seen as the asymptotic boundary of the tree. This conclusion complements previous results obtained when $\lambda\geq 1$.
\end{abstract}

\section{Introduction}

For random walks on finite graphs, a fundamental quantity to study is the cover time, that is, the time taken to visit every vertex. Moreover, for a sequence of graphs of increasing size, it is natural to ask how the cover time behaves asymptotically. Research in this direction has been extensive and we will not even attempt to give a comprehensive survey of such here. Rather, we focus on one particular example, namely, the binary tree. This simple graph is a valuable test case for probabilistic techniques and provides insight into the behaviour one might expect for related, but more complex, models.

To proceed with the discussion, let us start by introducing the main objects of interest. In particular, for $n\in\mathbb{N}$, we write $T_n$ for the binary tree of depth $n$, and $\rho_n$ for its root vertex. This is a graph with $2^{n+1}-1$ vertices, see Figure \ref{btfig} for an illustration of the binary tree of depth $4$. (Precise definitions are postponed to Section \ref{sec2}.) For a given $\lambda>0$, we consider the continuous-time Markov chain $((X^n_t)_{t\geq 0},({P}_x^n)_{x\in T_n})$ on $T_n$ that has unit mean holding times and jump probabilities given by: for $x\neq \rho_n$,
\begin{equation}\label{pndef}
P^n_x\left(X_1=y\right)=\left\{
                            \begin{array}{ll}
                              \frac{1}{\mathrm{deg}_{T_n}(x)-1+\lambda}, & \hbox{if $y$ is a child of $x$,} \\
                              \frac{\lambda}{\mathrm{deg}_{T_n}(x)-1+\lambda}, & \hbox{if $y$ is the parent of $x$,}
                            \end{array}
                          \right.
\end{equation}
where $\mathrm{deg}_{T_n}(x)$ is the usual graph degree of $x$ in $T_n$, and also
\[P^n_{\rho_n}\left(X_1=y\right)=\frac{1}{2},\qquad \hbox{if $y$ is a child of $\rho_n$.}\]
The process $X^n$ is the $\lambda$-biased random walk on $T_n$, and we will be interested in its cover time,
\begin{equation}\label{tcdef}
\tau_{\mathrm{cov}}(X^n):=\inf\left\{t\geq 0:\:X_{[0,t]}^n=T_n\right\},
\end{equation}
where we use the notation $X_{[0,t]}^n:=\{X_s^n:\:s\in[0,t]\}$, under the probability measure ${P}_{\rho_n}^n$.

In the cases when $\lambda=1$ or $\lambda>1$, the asymptotic behaviour of $\tau_{\mathrm{cov}}(X^n)$ is already well-understood. Indeed, in the unbiased case, i.e.\ $\lambda=1$, an increasingly detailed description of the scaling limit of the cover time was obtained through the works \cite{Ald,BRZ,BZ,DZ}, culminating in the result of \cite{CLS}, which demonstrated that, as $n\rightarrow\infty$,
\[\frac{\tau_{\mathrm{cov}}(X^n)}{n2^{n+1}}-n\log(2)+\log(n)\]
converges in distribution. (We have cited the statement exactly as given in \cite{CLS}, though note that the authors of that paper studied the so-called variable speed random walk, which has a jump rate of one along each edge. The corresponding statement for our definition of $X^n$ should incorporate an additional deterministic constant, see \cite[Remark 1]{CLS}.) The limit was described in terms of a Gumbel distribution, randomly scaled by a certain martingale limit of a Gaussian process defined on the infinite binary tree. This characterization is natural given the strong connections known to exist between Gaussian fields and the cover time, as was very clearly exhibited in the seminal paper \cite{DLP}, see also \cite{Ding}. Also, it has been seen that the results concerning the binary tree are closely related to those seen for random walk on the two-dimensional torus, see the discussion in \cite{CLS} for background.

\begin{figure}[t]
\begin{center}
\includegraphics[width=0.8\textwidth]{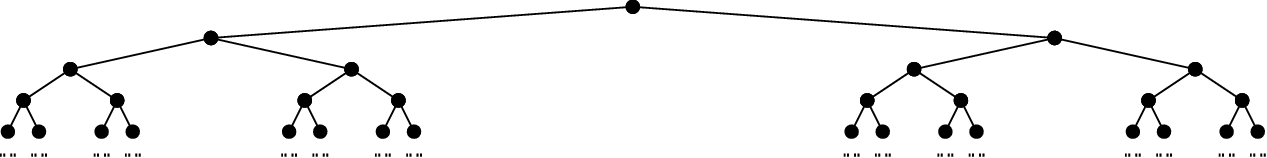}
\end{center}
\caption{The binary tree $T_n$ with $n=4$ and the `middle-thirds' Cantor set capturing the boundary of $T_n$ as $n\rightarrow\infty$.}\label{btfig}
\end{figure}

As for the $\lambda$-biased walk on the binary tree with $\lambda>1$, this was studied in \cite{Bai}. In fact, the paper \cite{Bai} handles the situation for $\lambda$-biased walks on supercritical Galton-Watson trees, but we will not consider the more general situation here. Specifically, \cite[Theorem 1.1]{Bai} gives (for a version of $T_n$ that is extended by adding a single edge to the root and running the random walk from this additional vertex) a distributional limit in terms of a Gumbel distribution for
\[\frac{(\lambda-1)(1-\frac{2}{\lambda})\tau_{\mathrm{cov}}(X^n)}{2\lambda^{n+1}}-n\log(2),\qquad \hbox{if $\lambda>2$},\]
\[\frac{\tau_{\mathrm{cov}}(X^n)}{(n+1)2^{n+2}}-n\log(2),\qquad \hbox{if $\lambda=2$},\]
\[\frac{(\lambda-1)(\frac{2}{\lambda}-1)\tau_{\mathrm{cov}}(X^n)}{2^{n+2}}-n\log(2),\qquad \hbox{if $\lambda\in(1,2)$}.\]
(Observe that a $\lambda=2$ threshold is natural, since this represents the phase transition between recurrence and transience for the $\lambda$-biased random walk on the infinite binary tree.)\footnote{We note there is a small error in the statement of \cite[Theorem 1.1]{Bai}. In particular, in the case $\lambda=m$ ($\lambda=2$ for us), the term $\sum_{i=0}^nZ_i$ should be replaced by $\sum_{i=0}^nZ_i/m^i$, which in our setting corresponds to the factor $n+1$.}

This leaves the case $\lambda\in (0,1)$, which has so far not been dealt with. It is the aim of this work to derive the first-order distributional asymptotics of $\tau_{\mathrm{cov}}(X^n)$ for $\lambda$ in this range. In contrast to the previously-described results, where the first-order behaviour is non-random, we will show that this is not the case when $\lambda<1$. (One can find a general result that gives a condition for concentration of the cover time on its mean in \cite{Aldcov}.) To explain this, we make two basic observations that will be made rigourous later in the article. Firstly, if we consider the electrical network naturally associated with $X^n$, in the case when $\lambda<1$, this converges to a compact tree-like metric space. Secondly, the invariant measure of $X^n$, scaled to be a probability measure, has a limit that is supported on the leaves of this metric space. By applying known results concerning Brownian motion on real trees and their time changes, we are able to show that (certain approximations to) $X^n$ converge in a suitable sense to a jump process that lives on the leaves, which can be represented by the `middle-thirds' Cantor set (again, see Figure \ref{btfig}, and also Lemma \ref{lem41} and Remark \ref{rem42} for further details). Moreover, with suitably careful analysis of the processes in question, one can check that the cover times also converge. Specifically, our main result is the following. A precise definition of the limiting process $X^\Sigma$ is given in Section \ref{sec3}; in particular, it is characterised by the Dirichlet form at \eqref{xsigform}.

\begin{theorem}\label{main}
Fix $\lambda\in(0,1)$, and let $X^n$ be the $\lambda$-biased random walk on the binary tree of depth $n$, started from its root $\rho_n$. It then holds that
\[\frac{2-\lambda}{4\lambda}\left(\frac{\lambda}{2}\right)^n\tau_{\mathrm{cov}}(X^n)\rightarrow \bar{\tau}_{\mathrm{cov}}(X^\Sigma),\]
in distribution as $n\rightarrow \infty$, where $X^\Sigma$ is a certain jump process on the middle-thirds Cantor set, started from an arbitrary point in that space, and $\bar{\tau}_{\mathrm{cov}}(X^\Sigma)$ is a $(0,\infty)$-valued random variable defined by setting
\begin{equation}\label{otau}
\bar{\tau}_{\mathrm{cov}}(X^\Sigma):=\inf\left\{t\geq 0:\:\overline{X^\Sigma_{[0,t]}}=\Sigma\right\}.
\end{equation}
(On the right-hand side above, we have used the notation $\overline{X^\Sigma_{[0,t]}}$ to represent the closure of ${X^\Sigma_{[0,t]}}$.) Moreover, writing $\|\cdot\|_p$ for the usual $L^p$ norm, we have for any $p\geq 1$ that
\[\frac{2-\lambda}{4\lambda}\left(\frac{\lambda}{2}\right)^n\left\|\tau_{\mathrm{cov}}(X^n)\right\|_{p}\rightarrow \left\|\bar{\tau}_{\mathrm{cov}}(X^\Sigma)\right\|_p,\]
where the limit takes a value in $(0,\infty)$.
\end{theorem}

\begin{rem}
We conjecture that $\bar{\tau}_{\mathrm{cov}}(X^\Sigma)$ is almost-surely equal to the infimum over times at which ${X^\Sigma_{[0,t]}}=\Sigma$, i.e.\ $\bar{\tau}_{\mathrm{cov}}(X^\Sigma)$ could be replaced by the usual notion of the cover time, defined analogously to \eqref{tcdef}. However, since $X^\Sigma$ is not continuous, this is a non-trivial claim, which our existing argument does not yield.
\end{rem}

Although our argument is relatively concise, it builds heavily on previous work developed for understanding stochastic processes on trees and fractals. Indeed, the construction of natural diffusions on trees goes back to work such as \cite{Kigden,Krebs}, and was completed in great generality in \cite{AEW}. As for the scaling limits of random walks to these processes, this was accomplished in a way suitable for treating critical Galton-Watson trees in \cite{Cr1,Cr2}, and much more broadly in \cite{ALW}. We note that even more general versions of these results, covering so-called resistance metric spaces were presented in \cite{Cr3,CHK}, applying the framework of resistance forms, as developed for understanding stochastic processes on fractals, see \cite{AOF,Kigres}. In this paper, we apply the extension of such results established in \cite{Noda}, which incorporates scaling limits of local times. The traces of processes on trees on their boundaries were studied in detail in \cite{Kigbound,Kigbound2} in the deterministic case; see also \cite{Tokushige} for the case of random trees. Furthermore, using a similar philosophy to that of this paper, it was recently shown that the cover times of random walks on critical Galton-Watson trees converge to that of Brownian motion on Aldous' continuum random tree \cite{ACMM}. However, in that paper, the argument depended on some key symmetries of the continuum random tree, which mean it is not readily transferable to other settings. The argument of this paper is much more straightforward, as it is possible to couple the random walks and the limiting process in a convenient way (in particular, we are able to show certain processes embedded in $X^n$ and $X^\Sigma$ are close), and this enables us to adopt a simpler approach along the lines of that originally proposed in \cite{CLLT}. Finally, we note that the moment convergence of Theorem \ref{main} is derived by applying the aforementioned connection between Gaussian fields and cover times, with appeals to results of \cite{DLP,Zhai} in particular. (In the proof of Theorem \ref{main}, we note some further historical references that are relevant in the application of these.)

\begin{rem}
(a) As also commented in \cite{Bai,CLS} for the $\lambda\geq 1$ case, when $\lambda<1$, the asymptotics of $\tau_{\mathrm{cov}}(X^n)$ are readily transferred to the discrete-time $\lambda$-biased random walk on $T_n$. In particular, the statement of Theorem \ref{main} applies to the latter process with no changes, since the difference between the time-scaling of the discrete and continuous processes occurs at a smaller order.\\
(b) The same argument should apply to more general deterministic trees, of the kind studied in \cite{Kigbound,Kigbound2}, or random ones, as considered in \cite{Bai,Tokushige}, at least under appropriate technical conditions. The limiting process would then be described as a  jump process on a more general (and possibly random) Cantor set. For the simplicity of presentation, we have chosen not to pursue this added generality here.
\end{rem}

To conclude the introduction, let us summarise the various regimes of the $\lambda$-biased random walk on a binary tree, including the contribution of the present work. In the following table, we present the first order growth rate of the cover time, as follows from Theorem \ref{main} and \cite{Bai,CLS}. We believe it is also informative to compare this with the corresponding growth rates for the resistance diameters and the total conductances of the electrical networks associated with $X^n$, $n\geq 1$ (see Section \ref{sec3} for definitions of these in our setting; these natural extend to all $\lambda>0$). For simplicity, we omit constants. In the final column, we describe the nature of the scaling limit for the cover time to first order; as described above, smaller order terms have been studied when $\lambda\geq 1$.
\smallskip

\begin{center}
\begin{tabular}{|c|c|c|c|c|}
  Regime & \begin{tabular}{c}Resistance\\diameter\end{tabular} & \begin{tabular}{c}Total\\conductance\end{tabular} & \begin{tabular}{c}Cover\\time\end{tabular} & Note\\
  \hline
$\lambda<1$       & 1           & $(2/\lambda)^n$ &  $(2/\lambda)^n$ & Random limit\\
$\lambda=1$       & $n$         & $(2/\lambda)^n$ & $n^22^n$         & Deterministic limit\\
$\lambda\in(1,2)$ & $\lambda^n$ & $(2/\lambda)^n$ & $n2^n$           & Deterministic limit\\
$\lambda=2$       & $\lambda^n$ & $n$             & $n^22^n$         & Deterministic limit\\
$\lambda>2$       & $\lambda^n$ & $1$             & $n\lambda^n$     & Deterministic limit\\
\end{tabular}
\end{center}
\smallskip

\noindent
Note in particular that for $\lambda<1$, the cover time is of the order of total conductance times resistance diameter, whereas in the cases covering $\lambda\geq 1$, there is an additional factor of $n$ (which is of the order of $\log (|T_n|)$). See also Figure \ref{fig2}. (We recall that total conductance times resistance diameter is always of the same order as the maximal expected hitting time of one vertex from another for a simple random walk on a graph, as can be seen from the commute time identity, see \cite[(2.17)]{CHK}, for example.) This divides the situation into the two extreme cases with regards the possible asymptotic behaviour of the cover time, as studied in \cite{Abe}, for example. (In the latter paper, the discrete time walk was considered, but as already commented, the first order results readily transfer to that setting.) Specifically, the $\lambda\geq 1$ case gives us examples of random walks whose cover time is maximal, and, as per the criterion of Aldous \cite[Theorem 2]{Aldcov}, the cover time concentrates on its mean. On the other hand, the $\lambda<1$ case gives us examples of random walks whose cover time is minimal, and, similarly to the conclusion of \cite[Proposition 1]{Aldcov} (which is for simple symmetric random walks), fails to concentrate on its mean.

\begin{figure}[t]
\begin{center}
\includegraphics[width=0.4\textwidth]{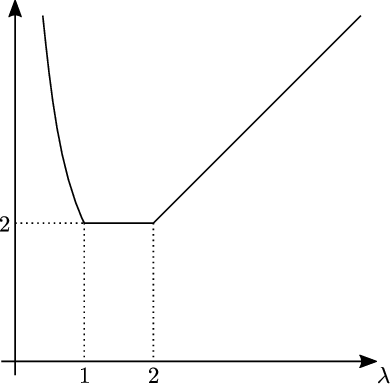}
\end{center}
\caption{A plot of the value of $\exp(\lim_{n\rightarrow\infty}n^{-1}\log\|\tau_{\mathrm{cov}}(X^n)\|_1$, which determines the geometric growth rate of the cover time, for the $\lambda$-biased random walk on a binary tree.}\label{fig2}
\end{figure}

The remainder of the article is organised as follows. In Section \ref{sec2}, we describe the spaces of the discussion precisely. This is followed in Section \ref{sec3} with an introduction to the associated stochastic processes and the connections between them. With these preparations in place, we then prove the main result, Theorem \ref{main}, in Section \ref{sec4}. Finally, the article concludes with Appendix \ref{appa}, in which we summarise some key background material on stochastic processes associated with resistance forms and their time changes. Note that we sometimes use a continuous variable $x$ where a discrete one is required; in such cases $x$ should be read as $\lfloor x\rfloor$ or $\lceil x\rceil$, as appropriate. Constants such as $C$ and $c$ take values in $(0,\infty)$, and may change from line to line.

\section{The underlying space}\label{sec2}

To prove our main result, it will be convenient to embed all the binary trees into a common metric space. The purpose of this section is to introduce this setting. The framework we describe is covered in greater generality in \cite{AEW} and \cite{Kigbound}.

First, for $n\geq 1$, define $\Sigma_n:=\{0,1\}^n$, let $\Sigma_0:=\{\emptyset\}$, and set $\Sigma_*:=\cup_{n\geq 0}\Sigma_n$. The binary tree $T_n$ can then be defined to be the graph with vertex set $\cup_{0\leq m\leq n}\Sigma_m$ and edges given by pairs of the form $\{i,ij\}$, where $i\in \cup_{0\leq m\leq n-1}\Sigma_m$, $j\in\{0,1\}$, and we write $ij$ for the concatenation of $i$ and $j$. (In a slight abuse of notation, we will often identify $T_n$ with its vertex set.) In particular, for all $n$, the root $\rho_n$ is given by $\emptyset$. Note that, with this construction, we have that $T_n\subseteq T_{n+1}$. In a similar way, we define the infinite binary tree $T_*$, which has vertex set $\Sigma_*$ and edges given by pairs of the form $\{i,ij\}$, where $i\in\Sigma_*$, $j\in\{0,1\}$.

We next extend $T_*$ to be a real tree (see \cite[Section 1.1]{AEW}, for example) by placing along each edge of the form $\{i,ij\}$, where $i\in\Sigma_n$, $j\in\{0,1\}$, a line segment of length $\lambda^n$, where $\lambda\in(0,1)$ is the fixed constant appearing in \eqref{pndef}. We write $T^o$ for the resulting space, and $d^o$ for the naturally-induced shortest path metric on $T^o$; this is the space of \cite[Example 1.12]{AEW} with $k=2$ and $c=\lambda$. We then let $(T,d)$ be the completion of $(T^o,d^o)$, and make the following basic observations about this space. (Since the proof is not difficult, we will be brief with the details.)

\begin{lem}\label{props}
(a) If $N(T,d,\varepsilon)$ is the minimal number of $\varepsilon$-balls needed to cover $(T,d)$, then
\[N(T,d,\varepsilon)\leq C\varepsilon^{-c},\]
where $C$ is a finite constant and $c=-\log(2/\lambda)/\log(\lambda)$.\\
(b) The metric space $(T,d)$ is a compact real tree.\\
(c) There is a one-to-one correspondence between the set $T\backslash T^o$ and $\Sigma=\{0,1\}^\mathbb{N}$, which is characterized by identifying $i\in\Sigma$ with $\lim_{n\rightarrow\infty}(i_1,\dots,i_n)$. (In what follows, we identify $T\backslash T^o$ and $\Sigma$.)\\
(d) The set $T\backslash\{x\}$ is connected if and only if $x\in \Sigma$, i.e.\ $\Sigma$ is the set of leaves of $T$.\\
(e) The space $(\Sigma,d)$ is topologically equivalent to the usual middle-thirds Cantor set, i.e.\ the unique compact set $K\subseteq\mathbb{R}$ satisfying
\[K=\bigcup_{i=0}^1\psi_i(K),\]
where $\psi_0(x)=\frac{1}{3}x$ and $\psi_1(x)=\frac{2}{3}+\frac{1}{3}x$, equipped with the usual Euclidean topology.
\end{lem}
\begin{proof}
(a) It is easy to check from the construction that $(T,d)$ is a real tree. For $x\in T$, denote by $T(x)$ the set consisting of those elements $y$ of $T$ for which the unique geodesic from $\emptyset$ to $y$ contains $x$. Note that for $x\in \Sigma_n$ and $y\in T(x)$, we have that
\begin{equation}\label{dxy}
d(x,y)\leq \sum_{m=n}^\infty\lambda^m=\frac{\lambda^n}{1-\lambda}.
\end{equation}
Hence
\[N\left(\cup_{x\in\Sigma_n}T(x),d,\frac{\lambda^n}{1-\lambda}\right)\leq 2^n.\]
Moreover, by considering each of the paths from $\emptyset$ to $x$, we have that
\[N\left(T\backslash\cup_{x\in\Sigma_n}T(x),d,\frac{\lambda^n}{1-\lambda}\right)\leq 2^n\left(\lambda^{-n} +1\right),\]
and so
\[N\left(T,d,\frac{\lambda^n}{1-\lambda}\right)\leq 2^n\left(\lambda^{-n} +2\right).\]
From this, the result readily follows.\\
(b) As already noted, $(T,d)$ is a real tree, and it is complete by definition. From part (a), we have that $(T,d)$ is totally bounded, and therefore it is also compact.\\
(c) One can readily check that $x\in T\backslash T^o$ if and only if
\[\{x\}=\bigcap_{n\geq 1}T(i_1,i_2,\dots,i_n)\]
for some $i=(i_1,i_2,\dots)\in \Sigma$, and, if the above equality holds, then $x=\lim_{n\rightarrow\infty}(i_1,\dots,i_n)$. These observations yield the desired correspondence.\\
(d) Again, this is straightforward to check from the construction.\\
(e) It is an elementary exercise to show that the map from $\Sigma$ to the middle-thirds Cantor set defined by
\begin{equation}\label{ilim}
i\mapsto \lim_{n\rightarrow\infty}\psi_{i_1}\circ\dots\circ\psi_{i_n}(0)
\end{equation}
is a homeomorphism, and this implies the result.
\end{proof}

We will equip $(T,d)$ with various Borel measures. (As an aid to the reader, we provide a sketch of their supports in Figure \ref{supports}.) To begin with, we will introduce a measure $\mu_T$ of full support by placing upon each edge of the form $\{i,ij\}$, where $i\in\Sigma_n$, $j\in\{0,1\}$, a copy of the one-dimensional Hausdorff measure, normalized to have mass equal to $(\lambda^{-1}-1)(2/\lambda)^{-n-1}$. Note that this measure has total mass given by
\[(\lambda^{-1}-1)\sum_{n=0}^\infty2^{n+1}\times\left(\frac{2}{\lambda}\right)^{-n-1}=(1-\lambda)\sum_{n=0}^\infty\lambda^n=1.\]
Hence $\mu_T$ is a probability measure on $T$. To define a measure that is supported on $\Sigma$, we first introduce notation $\Sigma(i):=\{ij:\:j\in \Sigma\}$, where $i\in\Sigma_*$. (Note that, in the notation of the proof of Lemma \ref{props}, $\Sigma(i)=\Sigma\cap T(i)$.) We then set $\mu_{\Sigma}$ to be the unique measure that satisfies
\begin{equation}\label{mus}
\mu_{\Sigma}\left(\Sigma(i)\right):=2^{-n},\qquad \forall i\in \Sigma_n,\:n\geq 0.
\end{equation}
Again, this is a probability measure. (If one were to map $\Sigma$ to the middle-thirds Cantor set in the way described at \eqref{ilim}, then this measure would correspond to the $\log(2)/\log(3)$-dimensional Hausdorff measure on the latter set.)

\begin{figure}[t]
\begin{center}
\includegraphics[width=0.65\textwidth]{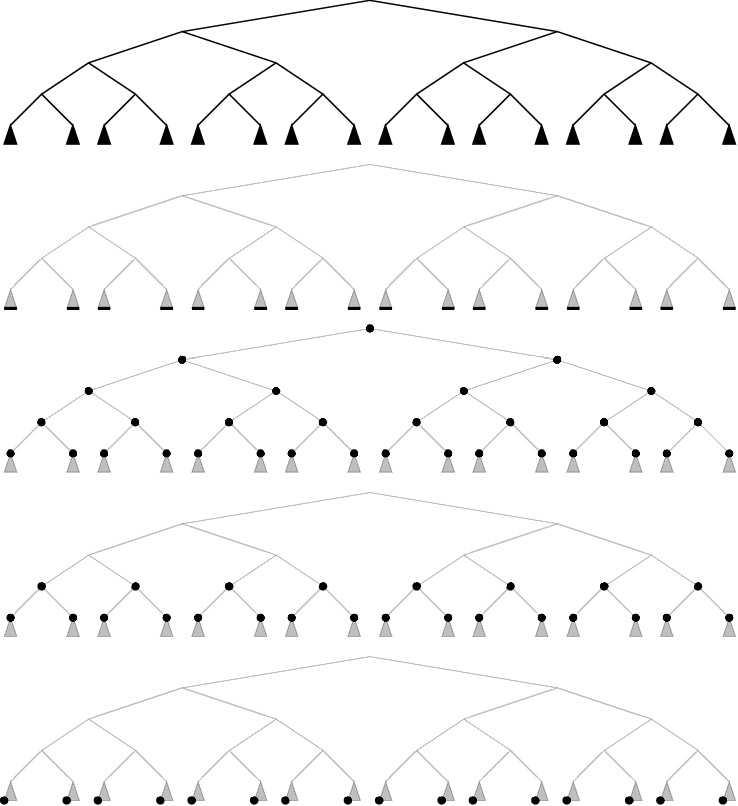}
\end{center}
\caption{A stylised representation of the supports of the measures $\mu_T$, $\mu_{\Sigma}$, $\mu_n$, $\bar{\mu}_n$ and $\tilde{\mu}_n$. In particular, the first sketch shows the support of $\mu_T$ as the whole of $T$, with the solid black triangles representing the parts of $T$ below vertices in $\Sigma_n$ (with $n=4$). The second shows the support of $\mu_{\Sigma}$ as the boundary set $\Sigma$. Third, the measure $\mu_n$ is supported on the vertices of $T_n$. Fourth, as will be introduced in the subsequent section, $\bar{\mu}_n$ is the restriction of $\mu_n$ to $\bar{\Sigma}_n=\cup_{m=n-\log n}^n\Sigma_m$. And finally, $\tilde{\mu}_n$ is the uniform measure on the subset $\tilde{\Sigma}_n$ of $\Sigma$, which has size $2^n$.
}\label{supports}
\end{figure}

To complete the section, we introduce various objects associated with the discrete spaces of interest. In particular, we consider each $T_n$ as an electrical network by equipping each edge of the form $\{i,ij\}$, where $i\in\Sigma_m$, $j\in\{0,1\}$, with a conductance given by
\[c(i,ij)=c(ij,i)=\lambda^{-m}.\]
We note that this is the inverse of the length of the line segments we inserted when defining $T^o$ above, and so the corresponding effective resistance between vertices $x$ and $y$ in $T_n$ is given by $d(x,y)$. (See \cite[Chapter 9]{LPW} for an introduction to electrical networks and their connection with reversible Markov chains, as we will discuss in the subsequent section.) We also introduce a measure on the vertices of $T_n$ by setting
\[\mu_{n}\left(\{x\}\right):=\sum_{y:\:\{x,y\}\in E_n}c(x,y),\]
where we have written $E_n$ for the edge set of $T_n$. In particular, we see that:
\[\mu_{n}\left(\{x\}\right)=\left\{
                              \begin{array}{ll}
                                2, & \hbox{if $x=\emptyset$,} \\
                                \lambda^{-(m-1)}+2\lambda^{-m}, & \hbox{if $x\in \Sigma_m$ for some $1\leq m\leq n-1$,} \\
                                \lambda^{-(n-1)}, & \hbox{if $x\in \Sigma_n$,}
                              \end{array}
                            \right.\]
and moreover,
\[\mu_{n}\left(T_n\right)=2\sum_{m=1}^n2^m\lambda^{-(m-1)}=\frac{4\lambda}{2-\lambda}\left(\left(\frac{2}{\lambda}\right)^n-1\right).\]
We will denote the constant on the right-hand side here by $b_n$. Finally, we define a set of leaves $\tilde{\Sigma}_n\subseteq\Sigma$ by setting
\[\tilde{\Sigma}_n:=\left\{ij:\:i\in \Sigma_n,\: j=(0,0,0,\dots)\right\},\]
and define a probability measure on this set by letting
\[\tilde{\mu}_n\left(\{x\}\right)=2^{-n},\qquad \forall x\in \tilde{\Sigma}_n.\]
The following result gives some basic properties of the measures $\mu_{n}$ and $\tilde{\mu}_n$.

\begin{lem}\label{lem22}
(a) It holds that
\[\lim_{n\rightarrow\infty}b_n^{-1}\mu_n\left(\bigcup_{m=0}^{n-\log n}\Sigma_m\right)=0.\]
(b) The measures $b_n^{-1}\mu_n$ and $\tilde{\mu}_n$ both converge weakly to $\mu_\Sigma$ as $n\rightarrow\infty$.
\end{lem}
\begin{proof}
Part (a) is straightforward. Indeed,
\[b_n^{-1}\mu_n\left(\bigcup_{m=0}^{n-\log n}\Sigma_m\right)\leq 2b_n^{-1}\sum_{m=1}^{n-\log n+1}2^m\lambda^{-(m-1)}=b_n^{-1}b_{n-\log n+1}\leq C\left(\frac{2}{\lambda}\right)^{-\log n}\rightarrow 0.\]

As for part (b), we start by considering $\tilde{\mu}_n$. To study this, we introduce a function $\tilde{\phi}_n:\Sigma\rightarrow\tilde{\Sigma}_n$ that, for each $i\in\Sigma_n$, maps $\Sigma(i)$ to $i(0,0,0,\dots)$. With this identification, we have
\[\tilde{\mu}_n=\mu_\Sigma\circ\tilde{\phi}_n^{-1}.\]
Hence, for any continuous function $f:T\rightarrow \mathbb{R}$,
\begin{equation}\label{jio}
\left|\tilde{\mu}_n(f)-\mu_\Sigma(f)\right|=\left|\int_\Sigma \left(f(x)-f\circ\tilde{\phi}_n(x)\right)\mu_\Sigma(dx)\right|\leq \sup_{x\in\Sigma}\left|f(x)-f\circ\tilde{\phi}_n(x)\right|.
\end{equation}
Now, from \eqref{dxy}, we have that
\[\sup_{x\in\Sigma}d\left(x,\tilde{\phi}_n(x)\right)\leq \frac{2\lambda^n}{1-\lambda}\rightarrow0.\]
Thus, since $f$ is uniformly continuous, the right-hand side of \eqref{jio} converges to 0, and we obtain that $\tilde{\mu}_n\rightarrow\mu_{\Sigma}$, as desired.

To handle $\mu_n$, we similarly define a function ${\phi}_n:\cup_{i\in\Sigma_{n-\log n}}T(i)\rightarrow\tilde{\Sigma}_{n-\log n}$ that, for each $i\in\Sigma_{n-\log n}$, maps $T(i)$ to $i(0,0,0,\dots)$, where $T(i)$ was defined in the proof of Lemma \ref{props}. Note that, for $i\in\Sigma_{n-\log n}$,
\[\mu_n\left(T(i)\right)=2\sum_{m=n-\log n+1}^n2^{m-n+\log n}\lambda^{-(m-1)}+\lambda^{-(n-\log n-1)}\]
does not depend on $i$. Hence, writing $c_n$ for this constant, we have that
\[\tilde{\mu}_n=c_n^{-1}2^{-(n-\log n)}\mu_n\circ\phi_n^{-1}.\]
It follows that, for any continuous function $f:T\rightarrow \mathbb{R}$,
\begin{eqnarray*}
\left|c_n^{-1}2^{-(n-\log n)}{\mu}_n(f)-\tilde{\mu}_n(f)\right|&\leq &c_n^{-1}2^{-(n-\log n)}\left|\int_{\cup_{i\in\Sigma_{n-\log n}}T(i)}\left(f(x)-f\circ{\phi}_n(x)\right)\mu_n(dx)\right|\\
&&\qquad+c_n^{-1}2^{-(n-\log n)}\int_{T\backslash\cup_{i\in\Sigma_{n-\log n}}T(i) }\left|f(x)\right|\mu_n(dx).
\end{eqnarray*}
Since
\[b_n^{-1}c_n=b_n^{-1}\left(\lambda^{-n+\log n}b_{\log n}+\lambda^{-(n-\log n-1)}\right)\sim \lambda^{-n+\log n}\left(\frac{2}{\lambda}\right)^{\log n-n}=2^{-(n-\log n)},\]
we obtain from part (a) that
\[c_n^{-1}2^{-(n-\log n)}\int_{T\backslash\cup_{i\in\Sigma_{n-\log n}}T(i) }\left|f(x)\right|\mu_n(dx)\rightarrow 0.\]
Moreover, since $b_n^{-1}\mu_n$ is a probability measure, we find that
\[b_n^{-1}\left|\int_{\cup_{i\in\Sigma_{n-\log n}}T(i)}\left(f(x)-f\circ{\phi}_n(x)\right)\mu_n(dx)\right|\leq \sup_{x\in\cup_{i\in\Sigma_{n-\log n}}T(i)}\left|f(x)-f\circ{\phi}_n(x)\right|,\]
and since the diameter of each $T(i)$ is bounded above by $2\lambda^n/(1-\lambda)$ the upper bound here converges to 0. In conclusion, we have proved that
\[\lim_{n\rightarrow\infty}b_n^{-1}\mu_n(f)=\lim_{n\rightarrow\infty}c_n^{-1}2^{-(n-\log n)}{\mu}_n(f)=\lim_{n\rightarrow\infty}\tilde{\mu}_n(f)=\mu_\Sigma(f),\]
where the existence and value of the limit is a consequence of the previous part of the proof. This yields that $b_n^{-1}{\mu}_n$ converges weakly to $\mu_\Sigma$, and thus completes the argument.
\end{proof}

\section{Stochastic processes on tree-like spaces}\label{sec3}

In order to couple the random walks on the binary trees of different levels and the limiting process on $\Sigma$ that we will soon introduce, we will construct them all as time changes of a canonical diffusion on $(T,d,\mu_T)$. In this section, we will define this principal object of interest, and explain its connection to the processes that appear in the statement of Theorem \ref{main}. The conclusions of this section depend on some fundamental statements concerning stochastic processes associated with resistance forms and their time changes; for the convenience of the reader, we present versions of the major results upon which we depend in Appendix \ref{appa}.

It is well-known that a compact real tree, equipped with a finite Borel measure of full support, is naturally associated with a strong Markov process that can be considered to be the Brownian motion on the space in question. See \cite{AEW,Kigden} for comprehensive background. In particular, we denote the process corresponding to $(T,d,\mu_T)$ by $X^T=((X^T_t)_{t\geq 0},(P^T_x)_{x\in T})$. As noted in \cite[Proposition 5.1]{Kigden}, $d$ can be considered to be a resistance metric on $T$ (see Definition \ref{resdefdef} below), and so $X^T$ is the process of Theorem \ref{hunt}. As per \cite[Theorem 1]{AEW}, this process is unique up to $\mu_T$-equivalence, has continuous sample paths and is $\mu_T$-symmetric. (In fact, as noted in the appendix, it is defined uniquely from any starting point $x\in T$.) Through the standard correspondence (see \cite[Theorem 7.2.1]{FOT}, for example), $X^T$ can be characterised in terms of a Dirichlet form $(\mathcal{E}_T,\mathcal{F}_T)$ on $L^2(T,\mu_T)$ that satisfies
\[d(x,y)^{-1}:=\inf\left\{\mathcal{E}_T(f,f):\:f\in\mathcal{F}_T,\:f(x)=0,\:f(y)=1\right\},\qquad \forall x,y\in T,\:x\neq y.\]
(See the discussion following Theorem \ref{hunt} for more background on the connection between $X^T$ and $(\mathcal{E}_T,\mathcal{F}_T)$.) Moreover, one can check the following result concerning the existence of local times.

\begin{lem}\label{ltlem}
The process $X^T$ admits jointly continuous local times $(L^T_t(x))_{x\in T,\:t\geq 0}$ such that, for any $x\in T$, it $P_x^T$-a.s.\ holds that
\begin{equation}\label{odf}
\int_0^tf(X_s)ds=\int_TL^T_t(y)\mu_T(dy)
\end{equation}
for any continuous $f:T\rightarrow\mathbb{R}$.
\end{lem}
\begin{proof}
Given the estimate of Lemma \ref{props}(a) (and the fact that $d$ is a resistance metric on $T$), the result follows immediately from Proposition \ref{localtimes}.
\end{proof}

Next, we proceed to introduce various time changes of $X^T$. Since the construction of these will all follow the pattern set out in Appendix \ref{appa}, we will describe the first in detail, and then simply list the remaining ones. To begin with, we introduce an additive functional $A^\Sigma=(A^\Sigma_t)_{t\geq 0}$ by setting
\begin{equation}\label{asig}
A^\Sigma_t:=\int_TL_t^T(x)\mu_\Sigma(dx),
\end{equation}
where $\mu_\Sigma$ was introduced at \eqref{mus}. Writing
\begin{equation}\label{alphasig}
\alpha^\Sigma(t):=\inf\left\{s\geq 0:\:A^\Sigma_s>t\right\}
\end{equation}
for its right-continuous inverse, we then set
\begin{equation}\label{xsigdef}
X^\Sigma_t:=X^T_{\alpha^\Sigma(t)}.
\end{equation}
By construction, we have that $X^\Sigma_0$ is distributed as the location of $X^T$ at the hitting time of $\Sigma$, which by the symmetry of the model is readily checked to have law $\mu_\Sigma$ under $P_\rho^T$. Moreover, by the trace theorem for Dirichlet forms \cite[Theorem 6.2.1]{FOT} (see Theorem \ref{timechangethm} below for the particular version that we apply here), the above definitions yield a strong Markov process $X^\Sigma=((X^\Sigma_t)_{t\geq 0},(P^\Sigma_x)_{x\in \Sigma})$, whose Dirichlet form $(\mathcal{E}_\Sigma,\mathcal{F}_\Sigma)$ on $L^2(\Sigma,\mu_\Sigma)$ is the trace of $(\mathcal{E}_T,\mathcal{F}_T)$ onto the latter space, i.e.\
\[\mathcal{F}_\Sigma:=\left\{f|_\Sigma:\:f\in \mathcal{F}_T\right\},\]
\begin{equation}\label{xsigform}
\mathcal{E}_\Sigma(f,f):=\inf\left\{\mathcal{E}_T(g,g):\:g\in\mathcal{F}_T:\:g|_\Sigma=f\right\}.
\end{equation}
We note that, by \cite[Theorem 8.4]{Kigres} (see Theorem \ref{timechangethm}(a)), the quadratic form $(\mathcal{E}_\Sigma,\mathcal{F}_\Sigma)$ is a resistance form in the sense of Definition \ref{resformdef}, and its associated resistance metric (see Definition \ref{resdefdef} and Theorem \ref{hunt}(a)) is the restriction of $d$ to $\Sigma$. Given the identification of $(\Sigma,d)$ with the middle-thirds Cantor set of Lemma \ref{props}(e), it is possible to consider $X^\Sigma$ to be a Markov process on this Cantor set. It is easy to check that the Dirichlet form $(\mathcal{E}_\Sigma,\mathcal{F}_\Sigma)$ is non-local, and so $X^\Sigma$ admits jumps; in fact, it is a pure jump process (as follows from \cite[Theorem 5.6]{Kigbound}). Various properties of $X^\Sigma$, such as estimates on its heat kernel, jump kernel and the moments of its displacement are given in \cite{Kigbound}. See also \cite{Baxter} for an earlier work in this direction. We summarise some basic properties of $X^T$ and $X^\Sigma$ in the following lemma.

\begin{lem}
(a) The process $X^T=((X^T_t)_{t\geq 0},(P^T_x)_{x\in T})$ associated with the measured resistance metric space $(T,d,\mu_T)$ is a strong Markov process that takes values in $C([0,\infty),T)$ (i.e.\ the space of continuous paths on $T$ equipped with the local uniform topology) and admits jointly continuous local times $(L^T_t(x))_{x\in T,\:t\geq 0}$ satisfying \eqref{odf}.\\
(b) The process $X^\Sigma=((X^\Sigma_t)_{t\geq 0},(P^\Sigma_x)_{x\in \Sigma})$ defined at \eqref{xsigdef} is the strong Markov process associated with the measured resistance metric space $(\Sigma,d,\mu_\Sigma)$. It is a pure jump process taking values in $D([0,\infty),\Sigma)$, i.e.\ the space of c\`{a}dl\`{a}g paths on $\Sigma$ equipped with the Skorohod $J_1$ topology.
\end{lem}

For our purposes, it is further useful to observe that, since Lemma \ref{props}(a) holds with $T$ replaced by $\Sigma$, by applying the same argument as used to prove Lemma \ref{ltlem} (i.e.\ an appeal to Proposition \ref{localtimes} below), we can conclude that $X^\Sigma$ admits jointly continuous local times $(L^\Sigma_t(x))_{x\in \Sigma,\:t\geq 0}$ satisfying the analogue to \eqref{odf}. Although we will not need it here, we note that, assuming $X^\Sigma$ is coupled with $X^T$ as above, these can be obtained from those for $X^T$ through the identity
\begin{equation}\label{ltequiv}
L^\Sigma_t(x)=L^T_{\alpha^\Sigma(t)}(x),\qquad \forall x\in \Sigma,
\end{equation}
cf.\ \cite[Lemma 3.4]{Cr1}. Applying the fact that $P_x^\Sigma(L_t^\Sigma(x)>0)=1$ for all $x\in \Sigma$ (which is a defining property of local times, see \cite[Definition 4.10]{Noda}, for example), the joint continuity of the local times $L^\Sigma$, the finiteness of hitting times of points by $X^\Sigma$ (as follows from the commute time identity, see \cite[(2.17)]{CHK}), and the compactness of the space, one may argue exactly as in the proof of \cite[Lemma 2.3]{Cr2} to deduce that: for every $x\in \Sigma$, it $P_x^\Sigma$-a.s.\ holds that
\begin{equation}\label{ltdiv}
\lim_{t\rightarrow\infty}\inf_{y\in \Sigma}L^\Sigma_t(y)=\infty.
\end{equation}
In particular, it follows that $\inf_{y\in \Sigma}L^\Sigma_t(y)>0$ for large $t$, $P_x^\Sigma$-a.s. Together with the fact that $X^\Sigma$ has c\`{a}dl\`{a}g sample paths in $\Sigma$, we readily obtain the following result, which gives the non-triviality of the limiting random variable that appears in the statement of Theorem \ref{main}.

\begin{lem}\label{tcovfin} For any $x\in\Sigma$, the random variable $\bar{\tau}_{\mathrm{cov}}(X^\Sigma)$ defined at \eqref{otau} takes values in $(0,\infty)$, $P_x^\Sigma$-a.s.
\end{lem}

As noted above, in the proof of our main result, we will also consider various other time changes of $X^T$. We summarise these in the table below. See also Figure \ref{links} for an overview of how these are connected. Each is obtained as a time change of $X^T$ via an additive functional of the form \eqref{asig}, where the relevant measure is the one appearing in the second column of the table. Specifically, we set
\[A^n_t:=\int_{T_n}L_t^T(x)\mu_n(dx),\qquad\bar{A}^n_t:=\int_{T_n}L_t^T(x)\bar{\mu}_n(dx),\qquad \tilde{A}^n_t:=\int_{\tilde{\Sigma}_n}L_t^T(x)\tilde{\mu}_n(dx),\]
where $\bar\mu_{n}$ is the restriction of $\mu_n$ to
\[\bar{\Sigma}_n:=\bigcup_{m=n-\log n}^n\Sigma_m.\]
(The sets $T_n$ and $\tilde{\Sigma}_n$ and measures $\mu_n$ and $\tilde{\mu}_n$ were introduced in the previous section.) The right-continuous inverses of these, which will be denoted by ${\alpha}^{n}$, $\bar{\alpha}^{n}$, $\tilde{\alpha}^{n}$, respectively, are defined as at \eqref{alphasig}, and the processes $X^n$, $\bar{X}^n$, $\tilde{X}^n$ similarly to \eqref{xsigdef}. By Theorem \ref{timechangethm}, each of these processes can be associated with a measured resistance metric space, as per the following table.
\smallskip

\begin{center}
\begin{tabular}{|c|c|c|c|c|}
   Process & Space & \begin{tabular}{c}Additive\\functional\end{tabular}& Inverse & Description\\
  \hline
  $X^T$ & $(T,d,\mu_T)$ & - & - & Brownian motion on $T$ \\
  $X^\Sigma$ & $(\Sigma,d,\mu_\Sigma)$ & $A^\Sigma$ & $\alpha^\Sigma$ & Jump process on Cantor set $\Sigma$ \\
  $X^n$ & $(T_n,d,\mu_n)$ & $A^n$ & $\alpha^n$ & $\lambda$-biased walk on $T_n$\\
  $\bar X^{n}$ & $(\bar{\Sigma}_n,d,\bar\mu_{n})$ & $\bar{A}^{n}$ & $\bar{\alpha}^{n}$ & $X^n$ observed on $\bar\Sigma_n$\\
  $\tilde X^{n}$ & $({\tilde{\Sigma}_n},d,\tilde{\mu}_n)$ & $\tilde{A}^{n}$ & $\tilde{\alpha}^{n}$ & Discrete approximation to $X^\Sigma$\\
\end{tabular}
\end{center}
\smallskip

\begin{figure}[t]
\begin{center}
 \begin{tikzcd}[row sep=+45pt, column sep =80pt]
        X^T \arrow[r, "\mbox{\scriptsize{$A^\Sigma$, \eqref{xsigdef}}}"] & X^\Sigma\arrow[r, shift left,"\mbox{\scriptsize{$\check{A}^n$, \eqref{checkan}}}"]\arrow[d,shift left, densely dotted,dash,"\mbox{\scriptsize{Lemma \ref{lem44}}}"] &\tilde{X}^n \arrow[dl,densely dotted,dash,"\mbox{\scriptsize{Lemma \ref{lem43}}}"]\arrow[l,shift left, densely dotted,dash,"\mbox{\scriptsize{Lemma \ref{lem46}}}"]\\
        X^n \arrow[r,shift left, "\mbox{\scriptsize{$\hat{A}^n$, \eqref{hatadef}}}"] & \bar{X}^n\arrow[u,shift left, "\mbox{\scriptsize{Lemma \ref{lem41}}}"]\arrow[l,shift left,densely dotted,dash,"\mbox{\scriptsize{Lemma \ref{lem42}}}"]
    \end{tikzcd}
    \end{center}
 \caption{Connections between the processes introduced in Section \ref{sec3}. The solid arrows show links between processes, either by construction or convergence. (We omit the arrows corresponding to $\bar{A}^n$ and $\tilde{A}^n$, which link $X^T$ to $\bar{X}^n$ and $\tilde{X}^n$, respectively, and also the convergence of $\tilde{X}^n$ to $X^\Sigma$, as given by Lemma \ref{lem41}.) The dotted lines represent links between the cover times of the various processes, as proved in Section \ref{sec4}. Specifically, Lemma \ref{lem42} shows the rescaled cover times of $X^n$ and $\bar{X}^n$ are close, Lemmas \ref{lem43} and \ref{lem44} demonstrate the rescaled cover time of $\bar{X}^n$ is asymptotically between that of $\tilde{X}^n$ and ${X}^\Sigma$, and Lemma \ref{lem46} establishes that the cover time of $\tilde{X}^n$ converges to that of ${X}^\Sigma$.}\label{links}
\end{figure}

To complete this section, we discuss a few properties of $X^n$, $\bar X^{n}$ and $\tilde X^{n}$ that we will need later. Firstly, from Remark \ref{resrem}, we see that $X^n$ is the continuous-time Markov process on $T_n$ with generator given by
\[\Delta_nf(x)=\sum_{y\in T_n}\frac{c_n(x,y)}{\mu_n(\{x\})}\left(f(y)-f(x)\right),\qquad \forall f:T_n\rightarrow \mathbb{R},\:x\in  T_n,\]
where $(c_n(x,y))_{x,y\in T_n}$ are the conductances associated with the resistance metric $d|_{T_n\times T_n}$. In particular, if $\{x,y\}$ is an edge in the graph $T_n$, then $c_n(x,y)=d(x,y)^{-1}$, and $c_n(x,y)=0$ otherwise. Similarly, $\bar X^{n}$ is the continuous-time Markov chain on $\bar{\Sigma}_n$ with generator given by
\[\bar{\Delta}_nf(x)=\sum_{y\in \bar{\Sigma}_n}\frac{\bar{c}_n(x,y)}{\bar{\mu}_n(\{x\})}\left(f(y)-f(x)\right),\qquad \forall f:\bar{\Sigma}_n\rightarrow \mathbb{R},\:x\in  \bar{\Sigma}_n,\]
where $(\bar{c}_n(x,y))_{x,y\in T_n}$ are the conductances associated with the resistance metric $d|_{\bar{\Sigma}_n\times \bar{\Sigma}_n}$. Note that these are the same as ${c}_n(x,y)$ for pairs of vertices connected by an edge in $T_n$, but will also be not identically zero for pairs of vertices in $\Sigma_{n-\log n}$. Indeed, the conductances between vertices in $\Sigma_{n-\log n}$ are determined by the part of the electrical network on $T_{n-\log n}$ with conductances given by $(c_n(x,y))_{x,y\in T_{n-\log n}}$. Moreover, since $\bar{\Sigma}_n\subseteq T_n$ and the resistance metrics and defining measures of $X^n$ and $\bar X^{n}$ agree on $\bar{\Sigma}_n$, by a further application of Theorem \ref{timechangethm}, we can obtain $\bar X^{n}$ as a time change of $X^{n}$ via the additive functional
\begin{equation}\label{hatadef}
\hat{A}^n_t=\int_0^t\mathbf{1}_{\{X^n_s\in\bar{\Sigma}_n\}}ds.
\end{equation}
That is, writing $\hat\alpha^n$ for the right-continuous inverse of $\hat{A}^n$, under $P^T_{\rho}$, we have that $(X^{n}_{\hat\alpha^n(t)})_{t\geq 0}$ is distributed as $\bar{X}^n$, started from a uniform point in $\Sigma_{n-\log n}$. (In fact, since it is the case that, started from a vertex $x\in \Sigma_{n-\log n}$, the process $X^n$ can hit any other vertex $y\in \Sigma_{n-\log n}$ without passing through the set $\bar{\Sigma}_n\backslash \{x,y\}$, the conductances $\bar{c}_n(x,y)$ will be non-zero for any pair $x,y\in \Sigma_{n-\log n}$.) Finally, we have that $\tilde{X}^n$ is the continuous-time Markov process on $\tilde{\Sigma}_n$ with generator given by
\[\tilde{\Delta}_nf(x)=\sum_{y\in \tilde{\Sigma}_n}\frac{\tilde{c}_n(x,y)}{\tilde{\mu}_n(\{x\})}\left(f(y)-f(x)\right),\qquad \forall f:\tilde{\Sigma}_n\rightarrow\mathbb{R},\:x\in  \tilde{\Sigma}_n,\]
where $(\tilde{c}_n(x,y))_{x,y\in T_n}$ are the conductances associated with the resistance metric $d|_{\tilde{\Sigma}_n\times \tilde{\Sigma}_n}$. In this case, since $\tilde{\Sigma}_n\subseteq \Sigma$ and the resistance metrics of $\tilde{X}^n$ and $X^\Sigma$ agree on $\tilde{\Sigma}_n$, we have from Theorem \ref{timechangethm} that $\tilde{X}^n$ can be given as a time change of $X^\Sigma$ via the additive functional
\begin{equation}\label{checkan}
\check{A}^n_t=\int_{\tilde{\Sigma}_n}L^\Sigma_t(x)\tilde\mu_n(dx),
\end{equation}
where $(L^\Sigma_t(x))_{x\in \Sigma,\:t\geq 0}$ are the local times of $X^\Sigma$, as defined above. Specifically, writing $\check\alpha^n$ for the right-continuous inverse of $\check{A}^n$, for any $x\in\tilde{\Sigma}_n$, we have that the law of $X^\Sigma_{\check\alpha^n(t)}$ under $P^\Sigma_x$ is the same as that of $\tilde{X}^n$, also started from $x$. The fact that $\tilde{X}^n$ can be embedded within $X^\Sigma$ in this way will be crucial in our later argument. (See Lemma \ref{lem46} in particular.)

\section{Proof of main result}\label{sec4}

The purpose of this section is prove Theorem \ref{main}. To this end, we will prove and then combine a number of lemmas about the processes introduced in the previous section. (To understand the basic structure of the argument, it might be helpful to refer to Figure \ref{links}.) Throughout, we will suppose that all the processes are coupled, each being constructed from the same realisation of $X^T$ by an appropriate time change. Moreover, unless otherwise noted, we will assume $X^T$ is started from $\rho=\emptyset$.

We start with a result that provides the link that we will require between the continuous-time Markov chains $\bar{X}^n$ and $\tilde{X}^n$ and limiting jump process $X^\Sigma$. It also gives some basic control on the time-scale of the cover time of $\bar{X}^n$. For the statement, we write
\[\bar{L}^n_{t}(x)=\frac{1}{\bar{\mu}^n(\{x\})}\int_{0}^t\mathbf{1}_{\{\bar{X}^n_s=x\}}ds,\qquad \tilde{L}^n_{t}(x)=\frac{1}{\tilde{\mu}^n(\{x\})}\int_{0}^t\mathbf{1}_{\{\tilde{X}^n_s=x\}}ds,\]
for the local times of the discrete-space processes. The convergence of the processes will be given in $D(\mathbb{R}_+,T)$, the space of c\`{a}dl\`{a}g functions on $T$ equipped with the usual Skorohod $J_1$-topology, and the convergence of local times will be given in $\hat{C}(T\times \mathbb{R}_+,\mathbb{R})$, the space of continuous functions from a closed subset of $T$ to $C(\mathbb{R}_+,\mathbb{R})$ equipped with the compact-convergence topology with variable domains of \cite[Definition 2.19]{Noda}. In particular, if $T^{(i)}$ are closed subsets of $T$ and $L^{(i)}:T^{(i)}\rightarrow C(\mathbb{R}_+,\mathbb{R})$ are continuous functions, $i=1,2$, then the distance between $L^{(1)}$ and $L^{(2)}$ in $\hat{C}(T\times \mathbb{R}_+,\mathbb{R})$ is given by the infimum over $\varepsilon$ such that: for each $x\in T^{(1)}$, there exists an element $y\in T^{(2)}$ satisfying
\begin{equation}\label{hatCdef}
\max\left\{d(x,y),d_{C(\mathbb{R}_+,\mathbb{R})}\left(L^{(1)}(x),L^{(2)}(y)\right)\right\}\leq \varepsilon,
\end{equation}
where $d_{C(\mathbb{R}_+,\mathbb{R})}$ is a metrisation of the topology of local uniform convergence, and, similarly, for each $y\in T^{(2)}$, there exists an element $x\in T^{(1)}$ satisfying the same bound.

\begin{lem}\label{lem41} (a) Under $P^T_\rho$, the law of
\[\left(\left(\bar{X}^n_{b_nt}\right)_{t\geq 0},\left(\bar{L}^n_{b_nt}(x)\right)_{x\in \bar{\Sigma}_n,\:t\geq 0}\right),\]
and also that of
\[\left(\left(\tilde{X}^n_t\right)_{t\geq 0},\left(\tilde{L}^n_{t}(x)\right)_{x\in \tilde{\Sigma}_n,\:t\geq 0}\right),\]
converges to that of $(X^\Sigma,L^\Sigma)$ in $D(\mathbb{R}_+,T)\times \hat{C}(T\times \mathbb{R}_+,\mathbb{R})$.\\
(b) Under $P^T_\rho$, the laws of $(b_n^{-1}\tau_{\mathrm{cov}}(\bar{X}^n))_{n\geq 1}$ form a tight sequence.
\end{lem}
\begin{proof}
From the definitions of the sets $\bar\Sigma_n$ and $\tilde{\Sigma}_n$, it is readily checked that
\[d_H\left(\bar\Sigma_n,\Sigma\right)\leq \sum_{m=n-\log n}^\infty\lambda^m\rightarrow0,\qquad d_H\left(\tilde\Sigma_n,\Sigma\right)\leq 2\sum_{m=n}^\infty\lambda^m\rightarrow0,\]
where we write $d_H$ for the Hausdorff distance between compact subsets of $T$ (with respect to the metric $d$). Moreover, we have from Lemma \ref{lem22} that $b_n^{-1}\bar\mu_n$ and $\tilde\mu_n$ both converge weakly to $\mu_\Sigma$ as $n\rightarrow\infty$. Additionally, Lemma \ref{props}(a) yields
\[N\left(\bar\Sigma_n,d,\varepsilon\right)\leq C\varepsilon^{-c},\qquad N\left(\tilde\Sigma_n,d,\varepsilon\right)\leq C\varepsilon^{-c}.\]
These assumptions allow us to apply \cite[Theorem 1.8]{Noda} to deduce part (a). We highlight that the main result of \cite{Noda} is written in the framework of Gromov-Hausdorff convergence of metric spaces, rather than Hausdorff convergence, since it is intended to handle spaces that are not necessarily embedded into a common space. However, in the proof, it is noted that Gromov-Hausdorff convergence implies the existence of a suitable embedding into a common space, and convergence is proved there. In our case, the canonical embedding of $\bar\Sigma_n$ and $\tilde{\Sigma}_n$ into $(T,d)$ by the identity map serves the desired purpose, and so we obtain the desired result.

For part (b), we start by noting that, because the topology of the convergence of part (a) is separable (see \cite[Theorem 2.28]{Noda}), we can apply the Skorohod representation theorem (see \cite[Theorem 4.30]{Kall}, for example) to couple all the random objects in question so that the convergence holds almost-surely. Recalling the definition of the topology on $\hat{C}(T\times \mathbb{R}_+,\mathbb{R})$ from its description around \eqref{hatCdef}, it is then an elementary exercise to establish that, almost-surely, for any $T>0$,
\[\lim_{\delta\rightarrow0}\limsup_{n\rightarrow\infty}\sup_{\substack{x_n\in\bar{\Sigma}_n,\:x\in \Sigma:\\d(x_n,x)\leq \delta}}\sup_{t\in[0,T]}\left|\bar{L}^n_{b_nt}(x_n)-L^\Sigma_t(x)\right|=0.\]
From this and the Hausdorff convergence of $\bar{\Sigma}_n$ to $\Sigma$, we deduce that
\[\inf_{x\in\bar{\Sigma}_n}\bar{L}^n_{b_nt}(x)\rightarrow \inf_{x\in \Sigma}L^\Sigma_t(x),\]
uniformly on compact time intervals, almost-surely. Now, we know from \eqref{ltdiv} that the right-hand side above is strictly positive for large $t$. Hence, it follows that, almost-surely, there exists a finite $t$ such that, for large $n$,
\[\inf_{x\in\bar{\Sigma}_n}\bar{L}^n_{b_nt}(x)>0.\]
Moreover, on the latter event, it holds that $\tau_{\mathrm{cov}}(\bar{X}^n)<b_nt$. In particular, we have shown that, on the coupled space,
\[\limsup_{n\rightarrow\infty }b_n^{-1}\tau_{\mathrm{cov}}(\bar{X}^n)<\infty,\]
almost-surely, and from this the tightness claim follows.
\end{proof}

\begin{rem}\label{rem42}
It is easy to see that one can not have convergence of $(X^n_{b_nt})_{t\geq0}$ to $X^\Sigma$ in $D(\mathbb{R}_+,T)$. Indeed, to travel from one side of the tree to the other, $X^n$ must pass through the root, and this has a strictly positive distance from $\Sigma$. However, since we have from Lemma \ref{lem22} that $b_n^{-1}\mu_n\rightarrow\mu_\Sigma$, by suitably adapting \cite{Cr3} along the lines of \cite[Theorem 1(ii)]{ALW}, we believe that it should be possible to check the finite-dimensional distributions converge.
\end{rem}

The next lemma in the sequence gives that $\tau_{\mathrm{cov}}(\bar{X}^n)$ is a good approximation for $\tau_{\mathrm{cov}}({X}^n)$.

\begin{lem}\label{lem42} For any $\varepsilon>0$,
\[P^T_{\rho}\left(b_n^{-1}\left|\tau_{\mathrm{cov}}({X}^n)-\tau_{\mathrm{cov}}(\bar{X}^n)\right|>\varepsilon\right)\rightarrow 0.\]
\end{lem}
\begin{proof}
From Lemma \ref{lem22}, we have that $b_n^{-1}\mu_n$ and $b_n^{-1}\bar\mu_n$ both converge weakly to $\mu_{\Sigma}$ as $n\rightarrow\infty$. Hence, applying the continuity of the local times $L^T$, we find that, $P^T_{\rho}$-a.s., for each $t\geq 0$,
\begin{equation}\label{anconv}
b_n^{-1}A^n_t=b_n^{-1}\int_{T_n}L^T_t(x)\mu_n(dx)\rightarrow A^{\Sigma}_t,
\end{equation}
and also
\begin{equation}\label{baranconv}
b_n^{-1}\bar{A}^n_t=b_n^{-1}\int_{T_n}L^T_t(x)\bar{\mu}_n(dx)\rightarrow A^{\Sigma}_t.
\end{equation}
Moreover, since the limit is a continuous, non-decreasing function, we have that the two convergence statements hold uniformly on compact time intervals.

We next claim that, $P^T_{\rho}$-a.s., for each $t\geq 0$,
\begin{equation}\label{taunclaim}
\sup_{n\geq 1}\alpha^n(b_nt)<\infty.
\end{equation}
Indeed, similarly to \eqref{ltdiv}, it is possible to check that, $P^T_{\rho}$-a.s.,
\[\lim_{t\rightarrow\infty}\inf_{y\in T}L^T_t(y)=\infty.\]
Hence, for any $t>0$, one can find an $s<\infty$ such that $A^\Sigma_s\geq t+1$. By \eqref{anconv}, it follows that, for large $n$, $b_n^{-1}A^n_s> t$, and this implies $\alpha^n(b_n t)\leq s$. The result at \eqref{taunclaim} is readily obtained from this.

Now, defining local times $L^n$ for $X^n$ by setting
\[{L}^n_{t}(x)=\frac{1}{{\mu}^n(\{x\})}\int_{0}^t\mathbf{1}_{\{{X}^n_s=x\}}ds,\]
one can check, similarly to \eqref{ltequiv}, that
\[{L}^n_{t}(x)=L^T_{\alpha^n(t)}(x),\qquad \forall x\in T_n,\:t\geq 0.\]
Hence, for the additive functional $\hat{A}^n$ defined at \eqref{hatadef}, it holds that, $P^T_{\rho}$-a.s., for each $K\geq 0$,
\begin{eqnarray*}
\sup_{t\in[0,K]}\left|b_n^{-1}\hat{A}^n_{b_nt}-t\right|&=&\sup_{t\in[0,K]}\left|b_n^{-1}\int_{\bar{\Sigma}_n}{L}^n_{b_nt}(x)\bar{\mu}_n(dx)-t\right|\\
&=&\sup_{t\in[0,K]}\left|b_n^{-1}\int_{\bar{\Sigma}_n}L^T_{\alpha^n(b_nt)}(x)\bar{\mu}_n(dx)-t\right|\\
&=&\sup_{t\in[0,K]}b_n^{-1}\left|\bar{A}^n_{\alpha^n(b_nt)}-{A}^n_{\alpha^n(b_nt)}\right|.
\end{eqnarray*}
Combining \eqref{anconv}, \eqref{baranconv} and \eqref{taunclaim}, it follows that, $P^T_{\rho}$-a.s., for each $K\geq 0$,
\[\sup_{t\in[0,K]}\left|b_n^{-1}\hat{A}^n_{b_nt}-t\right|\rightarrow 0\]
as $n\rightarrow\infty$. From this, it is an elementary exercise to deduce that, for the right-continuous inverse of $\hat{A}^n$, that is $\hat{\alpha}^n$, we have, $P^T_{\rho}$-a.s., for each $K\geq 0$,
\begin{equation}\label{pp}
\sup_{t\in[0,K]}\left|b_n^{-1}\hat{\alpha}^n(b_nt)-t\right|\rightarrow 0
\end{equation}
as $n\rightarrow\infty$.

Finally, we observe that since the cover time of $X^n$ must happen at a vertex in $\Sigma_n\subseteq\bar{\Sigma}_n$, it $P^T_{\rho}$-a.s.\ holds that
\begin{equation}\label{tat}
\tau_{\mathrm{cov}}(\bar{X}^n)=\hat{A}^n_{\tau_{\mathrm{cov}}({X}^n)}.
\end{equation}
Furthermore, since $\hat{A}^n$ is strictly increasing at $\tau_{\mathrm{cov}}({X}^n)$, we have that
\[\hat{\alpha}^n\left(\tau_{\mathrm{cov}}(\bar{X}^n)\right)={\tau_{\mathrm{cov}}({X}^n)}.\]
(Cf.\ \cite[Lemma 2.11]{ACMM}.) Consequently,
\[b_n^{-1}\left|\tau_{\mathrm{cov}}({X}^n)-\tau_{\mathrm{cov}}(\bar{X}^n)\right|=b_n^{-1}\left|\hat{\alpha}^n\left(\tau_{\mathrm{cov}}(\bar{X}^n)\right)-{\tau_{\mathrm{cov}}(\bar{X}^n)}\right|,\]
which in turn implies
\begin{align*}
\lefteqn{P^T_{\rho}\left(b_n^{-1}\left|\tau_{\mathrm{cov}}({X}^n)-\tau_{\mathrm{cov}}(\bar{X}^n)\right|>\varepsilon\right)}\\
&\leq P^T_{\rho}\left(b_n^{-1}\tau_{\mathrm{cov}}(\bar{X}^n)>K\right)+P^T_{\rho}\left(\sup_{t\in[0,K]}\left|b_n^{-1}\hat{\alpha}^n(b_nt)-t\right|>\varepsilon\right).
\end{align*}
From \eqref{pp}, the second term here converges to zero as $n\rightarrow\infty$ for any finite $K$. Moreover, from Lemma \ref{lem41}(b), we have that the $\limsup$ of the first term can be made arbitrarily small by taking $K$ large. This completes the proof.
\end{proof}

The following result establishes that $b_n^{-1}\tau_{\mathrm{cov}}(\bar{X}^n)$ is at most only a small amount greater than $\tau_{\mathrm{cov}}(\tilde{X}^n)$ asymptotically.

\begin{lem}\label{lem43} For any $\varepsilon>0$,
\[P^T_{\rho}\left(b_n^{-1}\tau_{\mathrm{cov}}(\bar{X}^n)-\tau_{\mathrm{cov}}(\tilde{X}^n)>\varepsilon\right)\rightarrow 0.\]
\end{lem}
\begin{proof} For a given point $x\in T$, it holds that $P_x^T(\inf\{t>0:\:L^T_t(x)>0\}=0)=1$. (This follows from the definition of local times, see \cite[Definition 4.10]{Noda}, for example.) From this, it readily follows that if $\tau_x(X^T)=\inf\{t\geq 0:\:X^T=x\}$ is the hitting time of $x$ by $X^T$ and $\tau_x(\bar{X}^n)$ is the similarly-defined hitting time of $x$ by $\bar{X}^n$, then we have that, for $x\in\bar\Sigma_n$,
\begin{equation}\label{a1}
\tau_x(\bar{X}^n)=\bar A^n_{\tau_x(X^T)},
\end{equation}
$P_\rho^T$-a.s. Since $\bar\Sigma_n$ is a finite set, this yields
\[\tau_{\mathrm{cov}}(\bar{X}^n)=\bar A^n_{\tau_{\mathrm{cov}}^{\bar{\Sigma}_n}(X^T)},\]
$P_\rho^T$-a.s., where
\[\tau_{\mathrm{cov}}^{\bar{\Sigma}_n}(X^T):=\inf\left\{t\geq 0:\:X^T_{[0,t]}\supseteq \bar{\Sigma}_n\right\}.\]
(This is the analogue of \eqref{tat} for the processes under consideration here.) For the same reasons, we also have that
\begin{equation}\label{twocovs}
\tau_{\mathrm{cov}}(\tilde{X}^n)=\tilde A^n_{\tau_{\mathrm{cov}}^{\tilde{\Sigma}_n}(X^T)},
\end{equation}
$P_\rho^T$-a.s., where
\begin{equation}\label{t2}
\tau_{\mathrm{cov}}^{\tilde{\Sigma}_n}(X^T):=\inf\left\{t\geq 0:\:X^T_{[0,t]}\supseteq \tilde{\Sigma}_n\right\}.
\end{equation}
Since $X^T$ has continuous sample paths, under $P_\rho^T$, we have that $\tau_{\mathrm{cov}}^{\bar{\Sigma}_n}(X^T)\leq \tau_{\mathrm{cov}}^{\tilde{\Sigma}_n}(X^T)$, and hence it follows that
\[b_n^{-1}\tau_{\mathrm{cov}}(\bar{X}^n)-\tau_{\mathrm{cov}}(\tilde{X}^n)\leq b_n^{-1}\bar A^n_{\tau_{\mathrm{cov}}^{\tilde{\Sigma}_n}(X^T)}-\tilde A^n_{\tau_{\mathrm{cov}}^{\tilde{\Sigma}_n}(X^T)}\leq \sup_{t\leq \tau_{\mathrm{cov}}^{\tilde{\Sigma}_n}(X^T)}\left|b_n^{-1}\bar A^n_t-\tilde{A}^n_t\right|.\]
(We highlight that the two arguments in the central expression here are both equal to $\tau_{\mathrm{cov}}^{\tilde{\Sigma}_n}(X^T)$.) Now, since the local times $L^T$ are jointly continuous, it follows from Lemma \ref{lem22} that
\[b_n^{-1}\bar A^n_t\rightarrow A^\Sigma_t, \qquad \tilde A^n_t\rightarrow A^\Sigma_t,\]
uniformly on compact time intervals, $P_\rho^T$-a.s. (Cf.\ \eqref{anconv}, \eqref{baranconv}.) Moreover, it is possible to check that $\tau_{\mathrm{cov}}^{\tilde{\Sigma}_n}(X^T)\leq \tau_{\mathrm{cov}}(X^T)<\infty$, $P_\rho^T$-a.s., where the finiteness claim can be checked by arguing as for Lemma \ref{tcovfin}. Hence we have established that, $P_\rho^T$-a.s.,
\[\sup_{t\leq \tau_{\mathrm{cov}}^{\tilde{\Sigma}_n}(X^T)}\left|b_n^{-1}\bar A^n_t-\tilde{A}^n_t\right|\rightarrow 0,\]
which confirms the result.
\end{proof}

From the convergence of processes $\bar{X}^n_{b_n\cdot}\rightarrow X^\Sigma$ of Lemma \ref{lem41}, it is straightforward to show the subsequent result, namely that $b_n^{-1}\tau_{\mathrm{cov}}(\bar{X}^n)$ is stochastically greater than $\bar\tau_{\mathrm{cov}}({X}^\Sigma)$ asymptotically.

\begin{lem}\label{lem44} For any $t\geq 0$, it holds that
\[\limsup_{n\rightarrow\infty}P_\rho^T\left(b_n^{-1}\tau_{\mathrm{cov}}(\bar{X}^n)\leq t\right)\leq P_\rho^T\left(\bar\tau_{\mathrm{cov}}({X}^\Sigma)\leq t\right).\]
\end{lem}
\begin{proof} The proof follows closely the argument of \cite[Corollary 7.3]{CLLT}. In particular, by combining Lemma \ref{lem41} and the Skorohod representation theorem (i.e.\ \cite[Theorem 4.30]{Kall}), we may assume that $\bar{X}^n$ and $X^\Sigma$ are coupled in such a  way that $(\bar{X}^n_{b_nt})_{t\geq 0}$ converges almost-surely to $X^\Sigma$ in $D(\mathbb{R}_+,T)$. Now, if $t<\bar\tau_{\mathrm{cov}}({X}^\Sigma)$, then there exists an open ball $B(x,\varepsilon)$ in $(T,d)$, with $x\in\Sigma$, $\varepsilon>0$,  such that $B(x,\varepsilon)\cap {X}^\Sigma_{[0,t]}=\emptyset$. It follows that, for any $t'<t$, $B(x,\varepsilon/2)\cap \bar{X}^n_{[0,b_nt']}=\emptyset$ for all large $n$. However, since $\bar{\Sigma}_n\rightarrow \Sigma$ as compact sets (with respect to the Hausdorff distance, see the proof of Lemma \ref{lem41}), it must hold that $B(x,\varepsilon/2)\cap\bar{\Sigma}_n\neq \emptyset$ for large $n$. Thus $\tau_{\mathrm{cov}}( \bar{X}^n)\geq b_nt'$. As $t'$ can be chosen to be arbitrarily close to $\bar\tau_{\mathrm{cov}}({X}^\Sigma)$, it follows that $\bar\tau_{\mathrm{cov}}({X}^\Sigma)\leq \liminf_{n\rightarrow\infty}b_n^{-1}\tau_{\mathrm{cov}}( \bar{X}^n)$, which yields the result.
\end{proof}

We now give an alternative characterisation of $\bar\tau_{\mathrm{cov}}({X}^\Sigma)$ in terms of a sequence of approximations. Similarly to \eqref{t2}, we define
\[\tau_{\mathrm{cov}}^{\tilde{\Sigma}_n}(X^\Sigma):=\inf\left\{t\geq 0:\:X^\Sigma_{[0,t]}\supseteq \tilde{\Sigma}_n\right\}.\]

\begin{lem}\label{lem45} It $P_\rho^T$-a.s.\ holds that
\[\bar\tau_{\mathrm{cov}}({X}^\Sigma)=\lim_{n\rightarrow\infty}\tau_{\mathrm{cov}}^{\tilde{\Sigma}_n}(X^\Sigma).\]
\end{lem}
\begin{proof}
First, note that $\tau_{\mathrm{cov}}^{\tilde{\Sigma}_n}(X^\Sigma)$ is increasing, and so it is $P_\rho^T$-a.s.\ possible to define
\[\tau:=\lim_{n\rightarrow\infty}\tau_{\mathrm{cov}}^{\tilde{\Sigma}_n}(X^\Sigma).\]
If $\tau=\infty$, then we clearly have $\tau\geq \bar\tau_{\mathrm{cov}}({X}^\Sigma)$. On the other hand, if we suppose $\tau<\infty$ and $t>\tau$, then it must hold that
\[X^{\Sigma}_{[0,t]}\supseteq\bigcup_{n\geq 1}\tilde\Sigma_n.\]
Since $\cup_{n\geq 1}\tilde\Sigma_n$ is dense in $\Sigma$, it follows that
\[\overline{X^\Sigma_{[0,t]}}=\Sigma,\]
and therefore $t\geq \bar\tau_{\mathrm{cov}}({X}^\Sigma)$. In particular, this establishes that $\tau\geq \bar\tau_{\mathrm{cov}}({X}^\Sigma)$.

Next, suppose that $t<\tau$. In this case, $t<\tau_{\mathrm{cov}}^{\tilde{\Sigma}_n}(X^\Sigma)$ for large $n$. Now, similarly to \eqref{a1}, one can check that, $P_\rho^T$-a.s.,
\[\tau_x(X^\Sigma)=A^\Sigma_{\tau_x(X^T)}\]
for any $x\in\Sigma$, where we denote by $\tau_x$ the hitting times of $x$ by the relevant processes, and hence we can assume that this is the case for all $x\in \cup_{n\geq 1}\tilde{\Sigma}_n$ for the realisation of $X^T$ that we are considering. From this, it readily follows that, for large $n$,
\[t<\tau_{\mathrm{cov}}^{\tilde{\Sigma}_n}(X^\Sigma)=A^\Sigma_{\tau_{\mathrm{cov}}^{\tilde{\Sigma}_n}(X^T)}\leq A^\Sigma_{\tau_{\mathrm{cov}}(X^T)}.\]
Applying $\alpha^\Sigma$ to both sides, this implies
\[\alpha^\Sigma(t)\leq \tau_{\mathrm{cov}}(X^T).\]
(Here, we have used that $A^\Sigma$ is strictly increasing at $\tau_{\mathrm{cov}}(X^T)$, and therefore $\tau_{\mathrm{cov}}(X^T)=\alpha^\Sigma(A^\Sigma_{\tau_{\mathrm{cov}}(X^T)})$; the claim holds because, under $P_\rho^T$, at the stopping time $\tau_{\mathrm{cov}}(X^T)$, we have that $X^T_{\tau_{\mathrm{cov}}(X^T)}\in\Sigma$.) Consequently, for any $s<\alpha^\Sigma(t)$, it is the case that $s<\tau_{\mathrm{cov}}(X^T)$, and so, by the continuity of $X^T$, there exists an open ball $B(x,\varepsilon)$ in $(T,d)$, with $x\in\Sigma$, $\varepsilon>0$, such that $B(x,\varepsilon)\cap {X}^T_{[0,s]}=\emptyset$. Now, take $u<t$. Since $A^\Sigma$ is continuous, we have that $\alpha^\Sigma$ is strictly increasing. So, setting $s=\alpha^\Sigma(u)$, we have that $s<\alpha^\Sigma(t)$. Hence, from our previous observation, we deduce that
\[B(x,\varepsilon)\cap {X}^\Sigma_{[0,u]}\subseteq B(x,\varepsilon)\cap {X}^T_{[0,s]}=\emptyset\]
for some open ball $B(x,\varepsilon)$ in $(T,d)$, with $x\in\Sigma$, $\varepsilon>0$. This implies $u\leq \bar\tau_{\mathrm{cov}}({X}^\Sigma)$. In conclusion, taking $u$ arbitrarily close to $\tau$, this yields $\tau\leq \bar\tau_{\mathrm{cov}}({X}^\Sigma)$, which completes the proof.
\end{proof}

For our final lemma, we establish that $\bar\tau_{\mathrm{cov}}({X}^\Sigma)$ arises as a limit of the cover times of the processes $\tilde{X}^n$.

\begin{lem}\label{lem46} It $P_\rho^T$-a.s.\ holds that
\[\lim_{n\rightarrow\infty}\tau_{\mathrm{cov}}(\tilde{X}^n)=\bar\tau_{\mathrm{cov}}({X}^\Sigma).\]
\end{lem}
\begin{proof} Similarly to \eqref{twocovs}, we have that
\begin{equation}\label{g1}
\tau_{\mathrm{cov}}(\tilde{X}^n)=\check A^n_{\tau_{\mathrm{cov}}^{\tilde{\Sigma}_n}(X^\Sigma)},
\end{equation}
where $\check{A}^n$ was defined at \eqref{checkan}. Moreover, by the joint continuity of $L^\Sigma$ and Lemma \ref{lem22}, it holds that
\[\check A^n_t\rightarrow \int_{\tilde{\Sigma}_n}L^\Sigma_t(x)\mu_\Sigma(dx)=t\]
uniformly on compact time intervals, $P_\rho^T$-a.s. The result follows by applying this convergence statement in conjunction with \eqref{g1} and Lemma \ref{lem45}.
\end{proof}

Putting the pieces together, we are able to give the proof of our main result.

\begin{proof}[Proof of Theorem \ref{main}]
Since $b_n^{-1}\sim \frac{2-\lambda}{4\lambda}(\lambda/2)^n$, from Lemma \ref{lem42} we immediately see that it will suffice to check the result with $\frac{2-\lambda}{4\lambda}(\lambda/2)^n\tau_{\mathrm{cov}}({X}^n)$ replaced by $b_n^{-1}\tau_{\mathrm{cov}}(\bar{X}^n)$. Given Lemma \ref{lem44}, this will follow if we can check that, for any $t>0$,
\begin{equation}\label{liminf}
\liminf_{n\rightarrow\infty}P_\rho^T\left(b_n^{-1}\tau_{\mathrm{cov}}(\bar{X}^n)\leq t\right)\geq P_\rho^T\left(\bar\tau_{\mathrm{cov}}({X}^\Sigma)< t\right).
\end{equation}
Now, from Lemma \ref{lem43}, we are able to check that, for any $t,\varepsilon>0$,
\begin{eqnarray*}
\liminf_{n\rightarrow\infty}P_\rho^T\left(b_n^{-1}\tau_{\mathrm{cov}}(\bar{X}^n)\leq t\right)&\geq&\liminf_{n\rightarrow\infty}P_\rho^T\left(b_n^{-1}\tau_{\mathrm{cov}}(\bar{X}^n)\leq t,\:b_n^{-1}\tau_{\mathrm{cov}}(\bar{X}^n)-\tau_{\mathrm{cov}}(\tilde{X}^n)\leq\varepsilon\right)\\
&\geq &\liminf_{n\rightarrow\infty}P_\rho^T\left(\tau_{\mathrm{cov}}(\tilde{X}^n)\leq t-\varepsilon,\:b_n^{-1}\tau_{\mathrm{cov}}(\bar{X}^n)-\tau_{\mathrm{cov}}(\tilde{X}^n)\leq\varepsilon\right)\\
&\geq &\liminf_{n\rightarrow\infty}P_\rho^T\left(\tau_{\mathrm{cov}}(\tilde{X}^n)\leq t-\varepsilon\right)\\
&&\qquad-\limsup_{n\rightarrow\infty}P_\rho^T\left(b_n^{-1}\tau_{\mathrm{cov}}(\bar{X}^n)-\tau_{\mathrm{cov}}(\tilde{X}^n)>\varepsilon\right)\\
&=&\liminf_{n\rightarrow\infty}P_\rho^T\left(\tau_{\mathrm{cov}}(\tilde{X}^n)\leq t-\varepsilon\right).
\end{eqnarray*}
From Lemma \ref{lem46}, the final expression here is bounded below by $P_\rho^T(\bar\tau_{\mathrm{cov}}({X}^\Sigma)<t-\varepsilon)$, which converges to $P_\rho^T(\bar\tau_{\mathrm{cov}}({X}^\Sigma)<t)$ as $\varepsilon\rightarrow 0$. Hence we have established \eqref{liminf}.

Let us highlight that, under $P_\rho^T$, the process $X^n$ is started from the root of $T_n$, that is $\rho_n$, and the process $X^\Sigma$ has initial distribution given by $\mu_\Sigma$. However, the symmetry of the model means that $\bar\tau_{\mathrm{cov}}({X}^\Sigma)$ has the same distribution under $P_x^\Sigma$ for any $x\in \Sigma$. Thus, we may replace the law of $\bar\tau_{\mathrm{cov}}({X}^\Sigma)$ under $P_\rho^T$ with that under $P_x^\Sigma$ for an arbitrary $x\in \Sigma$ to obtain the first claim of Theorem \ref{main}.

Given the distributional convergence result, to check the convergence of moments, it will suffice to check that, for each $p\geq 1$,
\begin{equation}\label{mombound}
\sup_{n\geq 1}\left(\frac{\lambda}{2}\right)^n\left\|\tau_{\mathrm{cov}}(X^n)\right\|_{p}<\infty.
\end{equation}
We will do this by using the connection between cover times and Gaussian fields. To this end, we first consider the quantity
\begin{equation}\label{gammadef}
\gamma_2\left(T_n,\sqrt{d}\right):=\inf\sup_{x\in T_n}\sum_{k\geq 0}2^{k/2}\sqrt{d(x,C_k)},
\end{equation}
where the infimum is taken over sequences $(C_k)_{k\geq 0}$ of subsets of $T_n$ with $|C_0|=1$ and $|C_k|\leq {2^{2^k}}$. By considering the possible choices of $C_0$, we see that \[\inf_{n\geq 1} \gamma_2\left(T_n,\sqrt{d}\right) \geq \inf_{n\geq 1}\inf_{x\in T_n}\sup_{y\in T_n}\sqrt{d(x,y)} \geq 1.\]
Moreover, we have that $\sup_{n\geq 1}\gamma_2(T_n,\sqrt{d})\leq \gamma_2(T,\sqrt{d})$, where $\gamma_2(T,\sqrt{d})$ is defined similarly to \eqref{gammadef}. Letting $C$ and $c$ be the constants of Lemma \ref{props}(a), we have for large $k$ that $2^{2kc}\leq 2^{2^k}$. Hence, for large $k$, we can take $C_k$ to be a $2^{-k}$-net of $(T,\sqrt{d})$ of minimal size, and this yields the finiteness of $\gamma_2(T,\sqrt{d})$. In conclusion, we have that
\[0<\inf_{n\geq 1} \gamma_2\left(T_n,\sqrt{d}\right)\leq \sup_{n\geq 1} \gamma_2\left(T_n,\sqrt{d}\right)<\infty.\]
From \cite[Theorem MM]{DLP}, we have that the quantity $\gamma_2(T_n,\sqrt{d})$ is comparable with the expected supremum of a related Gaussian field. In particular, applying the above bounds, \cite[Theorem MM]{DLP} yields that
\begin{equation}\label{gfest}
0<\inf_{n\geq 1} \mathbb{E}\sup_{x\in T_n}\eta_n(x)\leq \sup_{n\geq 1} \mathbb{E}\sup_{x\in T_n}\eta_n(x)<\infty,
\end{equation}
where, for each $n$, $(\eta_n(x))_{x\in T_n}$ is the centred Gaussian field on $T_n$ with $\eta_n(\rho_n)=0$ and covariances given by $\mathbb{E}((\eta_n(x)-\eta_n(y))^2)=d(x,y)$ (with $\mathbb{E}$ being the relevant expectation). We note that \cite[Theorem MM]{DLP} is a summary of earlier results, see \cite{Tala} in particular, with the definition of $\gamma_2$ that we choose being shown to be equivalent to the version of \cite{DLP} (up to constants) in \cite{Tala2}. Combining \eqref{gfest} with the facts that the diameters of the spaces $(T_n,d)$ are of order one and the total conductance $\mu_n(T_n)=b_n$ is of order $(2/\lambda)^n$, we obtain from \cite{Zhai}, which gives an exponential concentration result for the cover time of a random walk in terms of the expected supremum of the corresponding Gaussian field, that
\begin{equation}\label{concresult}
\sup_{n\geq 1}P_{\rho}^T\left(\tau_{\mathrm{cov}}(X^n)\geq u(2/\lambda)^n\right)\leq Ce^{-cu},\qquad \forall u\geq 1.
\end{equation}
(The main result of \cite{Zhai}, i.e.\ \cite[Theorem 1.1]{Zhai}, is given for unweighted graphs. However, the proof is given for random walks associated with electrical networks, and the above estimate follows from the final displayed equation of that. We further note that the result in question is a generalisation of the concentration result proved in \cite{Ding}.) From \eqref{concresult} we readily obtain \eqref{mombound}, and so the proof is complete.
\end{proof}

\appendix

\section{Background on resistance forms}\label{appa}

In this appendix, we present background material on stochastic processes associated with resistance forms and their time changes, as is required for this article. The fundamental theory of resistance forms was developed by Kigami in works such as \cite{AOF,Kigres}. As for the classical links between Dirichlet forms and stochastic processes, including time changes, these are treated in extensive detail in \cite{FOT}. Moreover, a study of time changes in the context of stochastic processes associated with resistance forms in particular is presented in \cite{CHK}. For the first part of the section, we closely follow the presentation of the introductory article \cite{crintro}.

Firstly, in order to introduce the notion of a resistance metric, we start by considering a finite, connected graph $G=(V,E)$, equipped with (strictly positive, symmetric) conductances $(c(x,y))_{\{x,y\}\in E}$. The corresponding effective resistance $(R(x,y))_{x,y\in V}$ is defined by setting $R(x,x)=0$ and, for $x\neq y$,
\begin{equation}\label{resdef}
R(x,y)^{-1}=\inf\left\{\frac12\sum_{x,y:\{x,y\}\in E}c(x,y)\left(f(x)-f(y)\right)^2:\:f(x)=1,f(y)=0\right\}.
\end{equation}
(See \cite{Barbook,LPW} for introductory treatments of the connections between electrical networks and random walks.) This definition was extended by Kigami to arbitrary sets as follows.

\begin{definition}[{\cite[Definition 2.3.2]{AOF}}]\label{resdefdef} Let $F$ be a set. A function $R:F\times F\rightarrow \mathbb{R}$ is a \emph{resistance metric} if, for every finite $V \subseteq F$, one can find a finite, connected graph with vertex set $V$, equipped with conductances, for which $R|_{V\times V}$ is the associated effective resistance.
\end{definition}

The sum that appears in the right-hand side of \eqref{resdef} represents the energy dissipation when the network in question is held at voltages according to $f$. Generalising the notion of electrical energy in the framework of Kigami are resistance forms. We next present a precise definition of such objects.

\begin{definition}[{\cite[Definition 2.3.1]{AOF}}]\label{resformdef} Let $F$ be a set. A pair $(\mathcal{E},\mathcal{F})$ is a resistance form on $X$ if it satisfies the following conditions.
\begin{description}
  \item[RF1] The symbol $\mathcal{F}$ represents a linear subspace of the collection of functions $\{f:\:F\rightarrow \mathbb{R}\}$ containing constants, and $\mathcal{E}$ is a non-negative symmetric quadratic form on $\mathcal{F}$ such that $\mathcal{E}(f,f)=0$ if and only if $f$ is constant on ${F}$.
  \item[RF2] Let $\sim$ be the equivalence relation on $\mathcal{F}$ defined by saying $f\sim g$ if and only if $f-g$ is constant on $F$. Then $(\mathcal{F}/\sim,\mathcal{E})$ is a Hilbert space.
  \item[RF3] If $x\neq y$, then there exists an $f\in \mathcal{F}$ such that $f(x)\neq f(y)$.
  \item[RF4] For any $x,y\in F$,
  \[\sup\left\{\frac{|f(x)-f(y)|^2}{\mathcal{E}(f,f)}:\:f\in\mathcal{F},\:\mathcal{E}(f,f)>0\right\}<\infty.\]
  \item[RF5] If $\bar{f}:=(f\wedge 1)\vee 0$, then $f\in\mathcal{F}$ and $\mathcal{E}(\bar{f},\bar{f})\leq \mathcal{E}(f,f)$ for any $f\in\mathcal{F}$.
\end{description}
\end{definition}

Whilst we will not discuss the specifics of this definition, we do highlight the following important connection between resistance metrics and forms, as well as an associated stochastic process. For simplicity of the statement, we restrict to the compact case, as that is all that is needed in this article. It is also possible to extend the result to locally compact spaces, though this requires a more careful treatment of the domain of the Dirichlet form.

\begin{theorem}[{\cite[Theorems 2.3.4, 2.3.6]{AOF}, \cite[Corollary 6.4 and Theorem 9.4]{Kigres}, \cite[Theorem 7.2.1]{FOT}}]\label{hunt} (a) Let $F$ be a set. There is a one-to-one correspondence between resistance metrics and resistance forms on $F$. This is characterised by the relation:
\[R(x,y)^{-1}=\inf\left\{\mathcal{E}(f,f):\:f(x)=1,f(y)=0\right\}\]
for $x\neq y$, and $R(x,x)=0$.\\
(b) Suppose $(F,R)$ is compact resistance metric space, and $\mu$ is a finite Borel measure on $F$ of full support. Then the corresponding resistance form $(\mathcal{E},\mathcal{F})$ is a regular Dirichlet form on $L^2(F,\mu)$, and so naturally associated with a $\mu$-symmetric Hunt process $X=((X_t)_{t\geq 0},(P_x)_{x\in F})$.
\end{theorem}

For those not familiar with Dirichlet forms, we remark that the connection between $(\mathcal{E},\mathcal{F})$ and $X$ of the above theorem can be seen through the generator of $X$. In particular, this is a non-positive definite self-adjoint operator $\Delta$ on $L^2(F,\mu)$ that yields the semigroup of $X$; specifically, the semigroup can be written $(e^{t\Delta})_{t\geq 0}$. Moreover, $\Delta$ also characterizes the Dirichlet form $(\mathcal{E},\mathcal{F})$ through the relation
\[\mathcal{E}(f,g)=-\int_F (\Delta f)  gd\mu,\qquad \forall f\in\mathcal{D},g\in\mathcal{F},\]
where $\mathcal{D}$ is the domain of $\Delta$. See \cite[Chapter 1]{FOT} for details. We further comment that, in general, when discussing the properties of Hunt processes associated with Dirichlet forms, these are only defined up to an exceptional set. However, in the case of resistance forms, it is known that any non-empty set has non-zero capacity (see \cite[Theorem 9.9]{Kigres}); it follows that the law $P_x$ is defined uniquely for every $x\in F$. For further properties of Hunt processes, including a detailed definition, we refer the reader to \cite[Section A.2]{FOT}, but we do note here that the $X$ of the Theorem \ref{hunt} will be a strong Markov process that has c\`{a}dl\`{a}g sample paths $P_x$-a.s.\ for every $x\in F$.

\begin{rem}\label{resrem}
If $(R(x,y))_{x,y\in V}$ is the effective resistance on a finite, connected graph $G=(V,E)$, equipped with (strictly positive, symmetric) conductances $(c(x,y))_{\{x,y\}\in E}$, and $\mu$ is a measure of full support on the set $V$, then the process $X$ of Theorem \ref{hunt} is the continuous-time Markov chain with generator given by, for $f:V\rightarrow\mathbb{R}$,
\[\Delta(f)(x):=\frac{1}{\mu(\{x\})}\sum_{y\in V:\{x,y\}\in E}c(x,y)(f(y)-f(x)).\]
\end{rem}

The notion of a time change of a stochastic process associated with Dirichlet form is studied extensively in \cite[Section 6.2]{FOT}. Here, we will only consider processes associated with compact resistance metric spaces, and it will always be the case that the processes we discuss admit jointly continuous local times in the sense of the following theorem. This enables a particularly simple and explicit presentation of the key result, see Theorem \ref{timechangethm} below.

\begin{prop}[{\cite[Proposition 4.15]{Noda}, cf. \cite[Lemma 2.2]{Cr2}}]\label{localtimes}
Let $(F,R)$ be a compact resistance metric space, and $\mu$ be a finite Borel measure on $F$ of full support. Suppose
\begin{equation}\label{cover}
\int_0^1\left(\log N(F,R^{1/4},r)\right)^{1/2}dr<\infty,
\end{equation}
where $N(F,R^{1/4},r)$ is the minimal number of $r$-balls needed to cover $(F,R^{1/4})$. It is then the case that the Hunt process $X$ of Theorem \ref{hunt} admits jointly continuous local times $(L_t(x))_{x\in F,\:t\geq0}$. Moreover, the local times satisfy the occupation density formula, i.e.\ it holds that, for all $x \in F$, $t \geq 0$ and all non-negative measurable functions $f:F \rightarrow\mathbb{R}_+$,
\[\int_{0}^tf(X_s)ds = \int_F f(y)L_t(y)\mu(dy),\qquad P_x\mbox{-a.s.}\]
\end{prop}

Now, let $\nu$ be a Borel measure on $F$ with support given by the closed set $\tilde{F}\subseteq F$. We complete the section by describing the time change of $X$ with respect to $\nu$. In particular, in the setting of Proposition \ref{localtimes}, it is possible to define an additive functional $(A_t)_{t\geq0}$ by setting
\[A_t:=\int_{\tilde{F}}L_t(y)\nu(dy).\]
Writing
\[\alpha(t):=\inf\left\{s\geq 0:\:A_s>t\right\}\]
for its right-continuous inverse, we then set
\begin{equation}\label{xnudef}
{X}^\nu_t:=X_{\alpha(t)};
\end{equation}
this is the time change of $X$ according to $\nu$. Just as $X$ is the process associated with $(F,R,\mu)$, it transpires that we have ${X}^\nu$ is the process associated with $(\tilde{F},\tilde{R},\nu)$, where $\tilde{R}$ is simply the restriction of $R$ to $\tilde{F}\times\tilde{F}$. This is the content of the final theorem of this section.

\begin{theorem}\label{timechangethm}
Let $(F,R)$ be a compact resistance metric space satisfying \eqref{cover}, and $\mu$ be a finite Borel measure on $F$ of full support. Moreover, suppose $\nu$ is a Borel measure on $F$ with support given by the closed set $\tilde{F}\subseteq F$. The following then hold.\\
(a) The pair $(\tilde{F},\tilde{R})$ is a compact resistance metric space. The associated resistance form, written  $(\tilde{\mathcal{E}},\tilde{\mathcal{F}})$, is characterized by:
\[\tilde{\mathcal{E}}(f,f)=\inf\left\{\mathcal{E}(g,g):\:g\in\mathcal{F},\:g|_{\tilde{F}}=f\right\},\]
\[\tilde{\mathcal{F}}=\left\{g|_{\tilde{F}}:\:g\in\mathcal{F}\right\}.\]
(b) Denote by $\tilde{X}=((\tilde{X}_t)_{t\geq 0},(\tilde{P}_x)_{x\in \tilde{F}})$ the $\nu$-symmetric Hunt process associated with $(\tilde{F},\tilde{R},\nu)$ by Theorem \ref{hunt}. For any $x\in F$, the law of $X^{\nu}$ (defined as at \eqref{xnudef}) under $P_x$ is given by that of $\tilde{X}$, started from $X_{\tau_{\tilde{F}}}$, where $\tau_{\tilde{F}}:=\inf\{t\geq 0:\:X_t\in\tilde{F}\}$ is the hitting time of $\tilde{F}$ by $X$. In particular, if $x\in \tilde{F}$, then the law of $X^{\nu}$ under $P_x$ is given by that of $\tilde{X}$ under $\tilde{P}_x$.
\end{theorem}
\begin{proof}
The first claim of part (a) is obvious from Definition \ref{resdefdef}. The remaining claim is (a simplification of) \cite[Theorem 8.4]{Kigres}. Part (b) is an application of \cite[Theorem 6.2.1]{FOT}.
\end{proof}

\section*{Acknowledgements}

The author was supported by JSPS Grant-in-Aid for Scientific Research (C) 24K06758 and the Research Institute for Mathematical Sciences, an International Joint Usage/Research Center located in Kyoto University. He would like to thank Jian Ding for making him aware of the reference \cite{Zhai}, which was used to check the convergence of moments in Theorem \ref{main}.

\bibliography{binarycover}

\providecommand{\bysame}{\leavevmode\hbox to3em{\hrulefill}\thinspace}
\providecommand{\MR}{\relax\ifhmode\unskip\space\fi MR }
\providecommand{\MRhref}[2]{%
  \href{http://www.ams.org/mathscinet-getitem?mr=#1}{#2}
}
\providecommand{\href}[2]{#2}
\begin{thebibliography}{10}

\bibitem{Abe}
Y.~Abe, \emph{Cover times for sequences of reversible {M}arkov chains on random
  graphs}, Kyoto J. Math. \textbf{54} (2014), no.~3, 555--576. \MR{3263552}

\bibitem{Ald}
D.~J. Aldous, \emph{Random walk covering of some special trees}, J. Math. Anal.
  Appl. \textbf{157} (1991), no.~1, 271--283. \MR{1109456}

\bibitem{Aldcov}
\bysame, \emph{Threshold limits for cover times}, J. Theoret. Probab.
  \textbf{4} (1991), no.~1, 197--211. \MR{1088401}

\bibitem{ACMM}
G.~Andriopoulos, D.~A. Croydon, V.~Margarint, and L.~Menard, \emph{On the cover
  time of {B}rownian motion on the {B}rownian continuum random tree}, preprint
  appears at arXiv:2410.03922, 2024.

\bibitem{AEW}
S.~Athreya, M.~Eckhoff, and A.~Winter, \emph{Brownian motion on
  {$\mathbb{R}$}-trees}, Trans. Amer. Math. Soc. \textbf{365} (2013), no.~6,
  3115--3150. \MR{3034461}

\bibitem{ALW}
S.~Athreya, W.~L\"ohr, and A.~Winter, \emph{Invariance principle for variable
  speed random walks on trees}, Ann. Probab. \textbf{45} (2017), no.~2,
  625--667. \MR{3630284}

\bibitem{Bai}
T.~Bai, \emph{On the cover time of {$\lambda$}-biased walk on supercritical
  {G}alton-{W}atson trees}, Stochastic Process. Appl. \textbf{130} (2020),
  no.~11, 6863--6879. \MR{4158805}

\bibitem{Barbook}
M.~T. Barlow, \emph{Random walks and heat kernels on graphs}, London
  Mathematical Society Lecture Note Series, vol. 438, Cambridge University
  Press, Cambridge, 2017. \MR{3616731}

\bibitem{Baxter}
M.~Baxter, \emph{Markov processes on the boundary of the binary tree},
  S\'eminaire de {P}robabilit\'es, {XXVI}, Lecture Notes in Math., vol. 1526,
  Springer, Berlin, 1992, pp.~210--224. \MR{1231996}

\bibitem{BRZ}
D.~Belius, J.~Rosen, and O.~Zeitouni, \emph{Barrier estimates for a critical
  {G}alton-{W}atson process and the cover time of the binary tree}, Ann. Inst.
  Henri Poincar\'e{} Probab. Stat. \textbf{55} (2019), no.~1, 127--154.
  \MR{3901643}

\bibitem{BZ}
M.~Bramson and O.~Zeitouni, \emph{Tightness for a family of recursion
  equations}, Ann. Probab. \textbf{37} (2009), no.~2, 615--653. \MR{2510018}

\bibitem{CLS}
A.~Cortines, O.~Louidor, and S.~Saglietti, \emph{A scaling limit for the cover
  time of the binary tree}, Adv. Math. \textbf{391} (2021), Paper No. 107974,
  78. \MR{4303734}

\bibitem{Cr1}
D.~A. Croydon, \emph{Convergence of simple random walks on random discrete
  trees to {B}rownian motion on the continuum random tree}, Ann. Inst. Henri
  Poincar\'e{} Probab. Stat. \textbf{44} (2008), no.~6, 987--1019. \MR{2469332}

\bibitem{Cr2}
\bysame, \emph{Scaling limits for simple random walks on random ordered graph
  trees}, Adv. in Appl. Probab. \textbf{42} (2010), no.~2, 528--558.
  \MR{2675115}

\bibitem{CLLT}
\bysame, \emph{Moduli of continuity of local times of random walks on graphs in
  terms of the resistance metric}, Trans. London Math. Soc. \textbf{2} (2015),
  no.~1, 57--79. \MR{3355578}

\bibitem{crintro}
\bysame, \emph{An introduction to stochastic processes associated with
  resistance forms and their scaling limits}, RIMS K\^{o}ky\^{u}roku
  \textbf{2030} (2017), paper no. 1.

\bibitem{Cr3}
\bysame, \emph{Scaling limits of stochastic processes associated with
  resistance forms}, Ann. Inst. Henri Poincar\'e{} Probab. Stat. \textbf{54}
  (2018), no.~4, 1939--1968. \MR{3865663}

\bibitem{CHK}
D.~A. Croydon, B.~M. Hambly, and Takashi Kumagai, \emph{Time-changes of
  stochastic processes associated with resistance forms}, Electron. J. Probab.
  \textbf{22} (2017), Paper No. 82, 41. \MR{3718710}

\bibitem{Ding}
J.~Ding, \emph{Asymptotics of cover times via {G}aussian free fields:
  bounded-degree graphs and general trees}, Ann. Probab. \textbf{42} (2014),
  no.~2, 464--496. \MR{3178464}

\bibitem{DLP}
J.~Ding, J.~R. Lee, and Y.~Peres, \emph{Cover times, blanket times, and
  majorizing measures}, Ann. of Math. (2) \textbf{175} (2012), no.~3,
  1409--1471. \MR{2912708}

\bibitem{DZ}
J.~Ding and O.~Zeitouni, \emph{A sharp estimate for cover times on binary
  trees}, Stochastic Process. Appl. \textbf{122} (2012), no.~5, 2117--2133.
  \MR{2921974}

\bibitem{FOT}
M.~Fukushima, Y.~Oshima, and M.~Takeda, \emph{Dirichlet forms and symmetric
  {M}arkov processes}, extended ed., De Gruyter Studies in Mathematics,
  vol.~19, Walter de Gruyter \& Co., Berlin, 2011. \MR{2778606}

\bibitem{Kall}
O.~Kallenberg, \emph{Foundations of modern probability}, second ed.,
  Probability and its Applications (New York), Springer-Verlag, New York, 2002.
  \MR{1876169}

\bibitem{Kigden}
J.~Kigami, \emph{Harmonic calculus on limits of networks and its application to
  dendrites}, J. Funct. Anal. \textbf{128} (1995), no.~1, 48--86. \MR{1317710}

\bibitem{AOF}
\bysame, \emph{Analysis on fractals}, Cambridge Tracts in Mathematics, vol.
  143, Cambridge University Press, Cambridge, 2001. \MR{1840042}

\bibitem{Kigbound}
\bysame, \emph{Dirichlet forms and associated heat kernels on the {C}antor set
  induced by random walks on trees}, Adv. Math. \textbf{225} (2010), no.~5,
  2674--2730. \MR{2680180}

\bibitem{Kigres}
\bysame, \emph{Resistance forms, quasisymmetric maps and heat kernel
  estimates}, Mem. Amer. Math. Soc. \textbf{216} (2012), no.~1015, vi+132.
  \MR{2919892}

\bibitem{Kigbound2}
\bysame, \emph{Transitions on a noncompact {C}antor set and random walks on its
  defining tree}, Ann. Inst. Henri Poincar\'e{} Probab. Stat. \textbf{49}
  (2013), no.~4, 1090--1129. \MR{3127915}

\bibitem{Krebs}
W.~B. Krebs, \emph{Brownian motion on the continuum tree}, Probab. Theory
  Related Fields \textbf{101} (1995), no.~3, 421--433. \MR{1324094}

\bibitem{LPW}
D.~A. Levin, Y.~Peres, and E.~L. Wilmer, \emph{Markov chains and mixing times},
  American Mathematical Society, Providence, RI, 2009, With a chapter by James
  G. Propp and David B. Wilson. \MR{2466937}

\bibitem{Noda}
R.~Noda, \emph{Convergence of local times of stochastic processes associated
  with resistance forms}, preprint appears at arXiv:2305.13224, 2023.

\bibitem{Tala}
M.~Talagrand, \emph{Regularity of {G}aussian processes}, Acta Math.
  \textbf{159} (1987), no.~1-2, 99--149. \MR{906527}

\bibitem{Tala2}
\bysame, \emph{Majorizing measures without measures}, Ann. Probab. \textbf{29}
  (2001), no.~1, 411--417. \MR{1825156}

\bibitem{Tokushige}
Y.~Tokushige, \emph{Jump processes on the boundaries of random trees},
  Stochastic Process. Appl. \textbf{130} (2020), no.~2, 584--604. \MR{4046511}

\bibitem{Zhai}
A.~Zhai, \emph{Exponential concentration of cover times}, Electron. J. Probab.
  \textbf{23} (2018), Paper No. 32, 22. \MR{3785402}

\end{thebibliography}
\bibliographystyle{amsplain}

\end{document}